\documentclass[twoside,11pt,reqno]{amsart}
\usepackage{amsmath,amsthm,amssymb,amstext,amsfonts,amscd}
\usepackage{graphicx}
\usepackage{multirow}

\renewcommand{\thefootnote}{\fnsymbol{footnote}}
\newtheorem{theorem}{Theorem}

\newtheorem{corollary}{Corollary}
\newtheorem{definition}{Definition}

\newtheorem{lemma}{Lemma}

\newtheorem{remark}{Remark}

\newtheorem{proposition}{Proposition}
\numberwithin{equation}{section}
\input{tcilatex}

\begin{document}
\title[Growth of $(\alpha ,\beta ,\gamma )$-order solutions of linear
differential equations]{Growth of $(\alpha ,\beta ,\gamma )$-order solutions
of linear differential equations with entire coefficients}
\author[B. Bela\"{\i}di and T. Biswas]{Benharrat Bela\"{\i}di$^{\ast }$ and
Tanmay Biswas}
\address{B. Bela\"{\i}di: Department of Mathematics, Laboratory of Pure and
Applied Mathematics, University of Mostaganem (UMAB), B. P. 227
Mostaganem-(Algeria).}
\email{benharrat.belaidi@univ-mosta.dz}
\address{T. Biswas: Rajbari, Rabindrapally, R. N. Tagore Road, P.O.-
Krishnagar, P.S. Kotwali, Dist-Nadia, PIN-741101, West Bengal, India}
\email{tanmaybiswas\_math@rediffmail.com}
\thanks{}

\begin{abstract}
The main aim of this paper is to study the growth of solutions of higher
order linear differential equations using the concepts of $(\alpha ,\beta
,\gamma )$-order and $(\alpha ,\beta ,\gamma )$-type. We obtain some results
which improve and generalize some previous results of Kinnunen \cite{13},
Long et al. \cite{L} as well as Bela\"{\i}di \cite{b3}, \cite{b5}.
\end{abstract}

\maketitle

\let\thefootnote\relax
\footnotetext{%
Corresponding author: benharrat.belaidi@univ-mosta.dz (Benharrat Bela\"{\i}di%
$\ast$)}

\section{\textbf{Introduction}}

\qquad Throughout this paper, we assume that the reader is familiar with the
fundamental results and the standard notations of the Nevanlinna value
distribution theory of entire and meromorphic functions and the theory of
complex linear differential equations which are available in \cite{1, 19, 3,
LY} and therefore we do not explain those in details. Let $f$ be an entire
function defined on $%
%TCIMACRO{\U{2102} }%
%BeginExpansion
\mathbb{C}
%EndExpansion
$. The maximum modulus function $M\left( r,f\right) $, the maximum term $\mu
\left( r,f\right) $ and central index $\nu (r,f)$ of $f(z)=\overset{+\infty }%
{\underset{n=0}{\sum }}a_{n}z^{n}$ on $\left\vert z\right\vert =r$ are
defined as $M\left( r,f\right) =\underset{\left\vert z\right\vert =r}{\max }%
\left\vert f\left( z\right) \right\vert $, $\mu \left( r,f\right) =\underset{%
n\geq 0}{\max }\left\{ |a_{n}|r^{n}\right\} $ and $\nu (r,f)=\max \left\{
n:\mu \left( r,f\right) =|a_{n}|r^{n}\right\} $ respectively. When $f$ is a
meromorphic\ function, one may introduce another function $T\left(
r,f\right) $ known as Nevanlinna's characteristic function of $f(z)$,
playing the same role as $M\left( r,f\right) $ defined by 
\begin{equation*}
T\left( r,f\right) =m\left( r,f\right) +N\left( r,f\right) ,
\end{equation*}%
with%
\begin{equation*}
m\left( r,f\right) =\frac{1}{2\pi }\int_{0}^{2\pi }\log ^{+}\left\vert
f\left( re^{i\theta }\right) \right\vert d\theta ,
\end{equation*}%
and%
\begin{equation*}
N\left( r,f\right) =\int_{0}^{r}\frac{n\left( t,f\right) -n\left( 0,f\right) 
}{t}dt+n\left( 0,f\right) \log r,
\end{equation*}%
where\ $\log ^{+}x=\max \left( 0,\log x\right) $\ for $x>0,$ and $n\left(
t,\ f\right) $\ is the number of poles of $f\left( z\right) $ lying in $%
\left\vert z\right\vert \leq t,$ counted according to their multiplicity. To
study the generalized growth properties of entire and meromorphic functions,
the concepts of different growth indicators such as the iterated $p$-order
(see \cite{13, DS}), the $(p,q)$-th order (see \cite{8, 9}), $(p,q)$-$%
\varphi $ order (see \cite{STX}) etc. are very useful and during the past
decades, several authors made close investigations on the generalized growth
properties of entire and meromorphic functions related to the above growth
indicators in some different directions. The theory of complex linear
differential equations has been developed since 1960s. Many authors have
investigated the complex linear differential equation%
\begin{equation}
L(f):=f^{(k)}(z)+A_{k-1}(z)f^{(k-1)}(z)+\cdots +A_{1}(z)f^{\prime
}(z)+A_{0}(z)f(z)=0  \label{1}
\end{equation}%
and achieved many valuable results when the coefficients $%
A_{0}(z),...,A_{k-1}(z)$ $(k\geq 2)$ in (\ref{1}) are entire functions of
finite order or finite iterated $p$-order or $(p,q)$-th order or $(p,q)$-$%
\varphi $ order; see (\cite{b1}, \cite{b2}, \cite{b3}, \cite{b}, \cite{bou}, 
\cite{c1}, \cite{20}, \cite{13}, \cite{19}, \cite{10}, \cite{11}, \cite{l}, 
\cite{STX}, \cite{t2}, \cite{t3}, \cite{15}).

\qquad Chyzhykov and Semochko \cite{CS} showed that both definitions of
iterated $p$-order and the $(p,q)$-th order have the disadvantage that they
do not cover arbitrary growth (see \cite[Example 1.4]{CS}). They used more
general scale, called the $\varphi $-order (see \cite{CS}) and the concept
of $\varphi $-order is used to study the growth of solutions of complex
differential equations which extend and improve many previous results (see 
\cite{b4, b5, CS, k}). Extending this notion, Long et al. \cite{L} recently
introduce the concepts of $[p,q]_{,\varphi }$-order and $[p,q]_{,\varphi }$%
-type (see \cite{L}) and obtain some interesting results which considerably
extend and improve some earlier results. For details\ one may see \cite{L}.

\qquad On the other hand, Mulyava et al. \cite{MST} have used the concept of 
$(\alpha ,\beta )$-order or generalized order of an entire function in order
to investigate the properties of solutions of a heterogeneous differential
equation of the second order and obtained several remarkable results. For
details\ one may see \cite{MST}.

\qquad The main aim of this paper is to study the growth of solutions of
higher order linear differential equations using the concepts of $(\alpha
,\beta ,\gamma )$-order and $(\alpha ,\beta ,\gamma )$-type. In fact, some
works relating to study the growth of solutions of higher order linear
differential equations using the concepts of $(\alpha ,\beta ,\gamma )$%
-order have been explored in \cite{b6} and \cite{b7}. In this paper, we
obtain some results which improve and generalize some previous results of
Long et al. \cite{L} as well as Bela\"{\i}di \cite{b3}, \cite{b5}. Further,
we give an answer to the problem of Chyzhykov and Semochko (\cite{CS},
Remark 1.11).

\section{\textbf{Definitions and Notations}}

\qquad {We denote the} Lebesgue linear measure of a set $F\subset \left[
0,+\infty \right) $ by $m\left( F\right) =\underset{F}{\int }dt,$ and the
logarithmic measure of a set $E\subset \left( 1,+\infty \right) $ by $%
m_{l}\left( E\right) =\underset{E}{\int }\frac{dt}{t}.$ Furthermore, let
throughout \ this paper $E$ represents a set of finite logarithmic measure, $%
I$ represents a set of infinite logarithmic measure and $F$ represents a set
of finite linear measure in proof of our results. For $x\in \lbrack
0,+\infty )$ and $k\in \mathbb{N}$ where $%
%TCIMACRO{\U{2115} }%
%BeginExpansion
\mathbb{N}
%EndExpansion
$ is the set of all positive integers, define{\ iterations of the
exponential and logarithmic functions as }$\exp ^{[k]}x=\exp (\exp
^{[k-1]}x) $ and $\log ^{[k]}x=\log (\log ^{[k-1]}x)$ {with convention that $%
\log ^{[0]}x=x$, $\log ^{[-1]}x=\exp {x}$, $\exp ^{[0]}x=x$ and $\exp
^{[-1]}x=\log x$.} Now, let $L$ be a class of continuous non-negative on $%
(-\infty ,+\infty )$ function $\alpha $ such that $\alpha (x)=\alpha
(x_{0})\geq 0$ for $x\leq x_{0}$ and $\alpha (x)\uparrow +\infty $ as $%
x_{0}\leq x\rightarrow +\infty $. We say that $\alpha \in L_{1}$, if $\alpha
\in L$ and $\alpha (a+b)\leq \alpha (a)+\alpha (b)+c$ for all $a,b\geq R_{0}$
and fixed $c\in (0,+\infty )$. Further, we say that $\alpha \in L_{2}$, if $%
\alpha \in L$ and $\alpha (x+O(1))=(1+o(1))\alpha (x)$ as $x\rightarrow
+\infty $. Finally, $\alpha \in L_{3}$, if $\alpha \in L$ and $\alpha
(a+b)\leq \alpha (a)+\alpha (b)$ for all $a,b\geq R_{0}$ i.e., $\alpha $ is
subadditive. Clearly $L_{3}\subset L_{1}$.

\qquad Particularly, when $\alpha \in L_{3}$, then one can easily verify
that $\alpha (mr)\leq m\alpha (r),$ $m\geq 2$ is an integer. Up to a
normalization, subadditivity is implied by concavity. Indeed, if $\alpha (r)$
is concave on $[0,+\infty )$ and satisfies $\alpha (0)\geq 0$, then for $%
t\in \lbrack 0,1]$,%
\begin{equation*}
\alpha (tx)=\alpha (tx+(1-t)\cdot 0)\geq t\alpha (x)+(1-t)\alpha (0)\geq
t\alpha (x)\text{,}
\end{equation*}%
so that by choosing $t=\frac{a}{a+b}$ or $t=\frac{b}{a+b}$,%
\begin{eqnarray*}
\alpha (a+b) &=&\frac{a}{a+b}\alpha (a+b)+\frac{b}{a+b}\alpha (a+b) \\
&\leq &\alpha \Big(\frac{a}{a+b}(a+b)\Big)+\alpha \Big(\frac{b}{a+b}(a+b)%
\Big) \\
&=&\alpha (a)+\alpha (b)\text{, \ \ }a,b\geq 0\text{.}
\end{eqnarray*}%
As a non-decreasing, subadditive and unbounded function, $\alpha (r)$
satisfies%
\begin{equation*}
\alpha (r)\leq \alpha (r+R_{0})\leq \alpha (r)+\alpha (R_{0})
\end{equation*}%
for any $R_{0}\geq 0$. This yields that $\alpha (r)\sim \alpha (r+R_{0})$ as 
$r\rightarrow +\infty $.

\qquad Now we add two conditions on $\alpha ,$ $\beta $ and $\gamma $: (i)
Always $\alpha \in L_{1},$ $\beta \in L_{2}$ and $\gamma \in L_{3}$; and
(ii) $\alpha (\log ^{[p]}x)=o(\beta (\log \gamma (x))),$ $p\geq 2,$ $\alpha
(\log x)=o(\alpha \left( x\right) )$ and $\alpha ^{-1}(kx)=o\left( \alpha
^{-1}(x)\right) $ $\left( 0<k<1\right) $ as $x\rightarrow +\infty $.

\qquad Throughout this paper, we assume that $\alpha ,$ $\beta $ and $\gamma 
$ always satisfy the above two conditions unless otherwise specifically
stated.

\qquad Heittokangas et al. \cite{HWWY} have introduced a new concept of $%
\varphi $-order of entire and meromorphic function considering $\varphi $ as
subadditive function. For details\ one may see \cite{HWWY}. Extending this
notion, recently Bela\"{\i}di et al. \cite{b6} introduce the definition of
the $(\alpha ,\beta ,\gamma )$-order of a meromorphic function in the
following way:

\begin{definition}
\label{d1}(\cite{b6}) The $(\alpha ,\beta ,\gamma )$-order denoted by $%
\sigma _{(\alpha ,\beta ,\gamma )}[f]$ of a meromorphic function $f$ is
defined by%
\begin{equation*}
\sigma _{(\alpha ,\beta ,\gamma )}[f]=\underset{r\rightarrow +\infty }{\lim
\sup }\frac{\alpha \left( \log T\left( r,f\right) \right) }{\beta \left(
\log \gamma \left( r\right) \right) }\text{.}
\end{equation*}
\end{definition}

\qquad By the inequality in \cite{1} 
\begin{equation*}
T(r,f)\leq \log ^{+}M(r,f)\leq \frac{R+r}{R-r}T(R,f)\text{ }(0<r<R)
\end{equation*}%
for an entire function $f(z)$, one can easily verify that \cite{b6} 
\begin{equation*}
\sigma _{(\alpha ,\beta ,\gamma )}[f]=\underset{r\rightarrow +\infty }{\lim
\sup }\frac{\alpha \left( \log T\left( r,f\right) \right) }{\beta \left(
\log \gamma \left( r\right) \right) }=\underset{r\rightarrow +\infty }{\lim
\sup }\frac{\alpha (\log ^{[2]}M(r,f))}{\beta \left( \log \gamma \left(
r\right) \right) }\text{.}
\end{equation*}

\begin{proposition}
\label{p1}(\cite{b6}) If $f$ is an entire function, then%
\begin{equation*}
\sigma _{(\alpha (\log ),\beta ,\gamma )}[f]=\underset{r\rightarrow +\infty }%
{\lim \sup }\frac{\alpha (\log ^{[2]}T(r,f))}{\beta \left( \log \gamma
\left( r\right) \right) }=\underset{r\rightarrow +\infty }{\lim \sup }\frac{%
\alpha (\log ^{[3]}M(r,f))}{\beta \left( \log \gamma \left( r\right) \right) 
}\text{,}
\end{equation*}%
where $(\alpha (\log ),\beta ,\gamma )$-order denoted by $\sigma _{(\alpha
(\log ),\beta ,\gamma )}[f]$.
\end{proposition}

\qquad Now to compare the relative growth of two meromorphic functions
having same non zero finite $(\alpha ,\beta ,\gamma )$-order, one may
introduce the definition of $(\alpha ,\beta ,\gamma )$-type in the following
manner:

\begin{definition}
\label{d2}The $(\alpha ,\beta ,\gamma )$-type denoted by $\tau _{(\alpha
,\beta ,\gamma )}[f]$ of a meromorphic function $f$ with $0<\sigma _{(\alpha
,\beta ,\gamma )}[f]<+\infty $ is defined by%
\begin{equation*}
\tau _{(\alpha ,\beta ,\gamma )}[f]=\underset{r\rightarrow +\infty }{\lim
\sup }\frac{\exp (\alpha \left( \log T\left( r,f\right) \right) )}{\left(
\exp \left( \beta \left( \log \gamma \left( r\right) \right) \right) \right)
^{\sigma _{(\alpha ,\beta ,\gamma )}[f]}}\text{.}
\end{equation*}%
If $f$ is an entire function with $\sigma _{(\alpha ,\beta ,\gamma )}[f]\in
(0,+\infty )$, then the $(\alpha ,\beta ,\gamma )$-type of $f$ is defined by%
\begin{equation*}
\tau _{(\alpha ,\beta ,\gamma ),M}[f]=\underset{r\rightarrow +\infty }{\lim
\sup }\frac{\exp (\alpha (\log ^{[2]}M(r,f)))}{\left( \exp \left( \beta
\left( \log \gamma \left( r\right) \right) \right) \right) ^{\sigma
_{(\alpha ,\beta ,\gamma )}[f]}}\text{.}
\end{equation*}
\end{definition}

\qquad Similar to Definition \ref{d2}, we can also define the $(\alpha (\log
),\beta ,\gamma )$-type of a meromorphic function $f$ in the following way:

\begin{definition}
\label{d3}The $(\alpha (\log ),\beta ,\gamma )$-type denoted by $\tau
_{(\alpha (\log ),\beta ,\gamma )}[f]$ of a meromorphic function $f$ with $%
0<\sigma _{(\alpha (\log ),\beta ,\gamma )}[f]<+\infty $ is defined by%
\begin{equation*}
\tau _{(\alpha (\log ),\beta ,\gamma )}[f]=\underset{r\rightarrow +\infty }{%
\lim \sup }\frac{\exp (\alpha (\log ^{[2]}T(r,f)))}{\left( \exp \left( \beta
\left( \log \gamma \left( r\right) \right) \right) \right) ^{\sigma
_{(\alpha (\log ),\beta ,\gamma )}[f]}}\text{.}
\end{equation*}%
If $f$ is an entire function with $\sigma _{(\alpha (\log ),\beta ,\gamma
)}[f]\in (0,+\infty )$, then the $(\alpha ,\beta ,\gamma )$-type of $f$ is
defined by%
\begin{equation*}
\tau _{(\alpha (\log ),\beta ,\gamma ),M}[f]=\underset{r\rightarrow +\infty }%
{\lim \sup }\frac{\exp (\alpha (\log ^{[3]}M(r,f)))}{\left( \exp \left(
\beta \left( \log \gamma \left( r\right) \right) \right) \right) ^{\sigma
_{(\alpha (\log ),\beta ,\gamma )}[f]}}\text{.}
\end{equation*}
\end{definition}

\begin{proposition}
\label{p2}(\cite{b6}) Let $f_{1},$ $f_{2}$ be nonconstant meromorphic
functions with $\sigma _{(\alpha ,\beta ,\gamma )}[f_{1}]$ and $\sigma
_{(\alpha ,\beta ,\gamma )}[f_{2}]$ as their $(\alpha ,\beta ,\gamma )$%
-order. Then\newline
(i) $\sigma _{(\alpha ,\beta ,\gamma )}[f_{1}\pm f_{2}]\leq \max \{\sigma
_{(\alpha ,\beta ,\gamma )}[f_{1}],$ $\sigma _{(\alpha ,\beta ,\gamma
)}[f_{2}]\};$\newline
(ii) $\sigma _{(\alpha ,\beta ,\gamma )}[f_{1}\cdot f_{2}]\leq \max \{\sigma
_{(\alpha ,\beta ,\gamma )}[f_{1}],$ $\sigma _{(\alpha ,\beta ,\gamma
)}[f_{2}]\};$\newline
(iii) If $\sigma _{(\alpha ,\beta ,\gamma )}[f_{1}]\neq \sigma _{(\alpha
,\beta ,\gamma )}[f_{2}]$, then $\sigma _{(\alpha ,\beta ,\gamma )}[f_{1}\pm
f_{2}]=\max \{\sigma _{(\alpha ,\beta ,\gamma )}[f_{1}],$ $\sigma _{(\alpha
,\beta ,\gamma )}[f_{2}]\};$\newline
(iv) If $\sigma _{(\alpha ,\beta ,\gamma )}[f_{1}]\neq \sigma _{(\alpha
,\beta ,\gamma )}[f_{2}]$, then $\sigma _{(\alpha ,\beta ,\gamma
)}[f_{2}\cdot f_{2}]=\max \{\sigma _{(\alpha ,\beta ,\gamma )}[f_{1}],$ $%
\sigma _{(\alpha ,\beta ,\gamma )}[f_{2}]\}$.
\end{proposition}

\begin{proposition}
\label{p3}(\cite{b6})Let $f_{1},$ $f_{2}$ be nonconstant meromorphic
functions with $\sigma _{(\alpha (\log ),\beta ,\gamma )}[f_{1}]$ and $%
\sigma _{(\alpha (\log ),\beta ,\gamma )}[f_{2}]$ as their $(\alpha (\log
),\beta ,\gamma )$-order. Then\newline
(i) $\sigma _{(\alpha (\log ),\beta ,\gamma )}[f_{1}\pm f_{2}]\leq \max
\{\sigma _{(\alpha (\log ),\beta ,\gamma )}[f_{1}],$ $\sigma _{(\alpha (\log
),\beta ,\gamma )}[f_{2}]\};$\newline
(ii) $\sigma _{(\alpha (\log ),\beta ,\gamma )}[f_{1}\cdot f_{2}]\leq \max
\{\sigma _{(\alpha (\log ),\beta ,\gamma )}[f_{1}],$ $\sigma _{(\alpha (\log
),\beta ,\gamma )}[f_{2}]\};$\newline
(iii) If $\sigma _{(\alpha (\log ),\beta ,\gamma )}[f_{1}]\neq \sigma
_{(\alpha (\log ),\beta ,\gamma )}[f_{2}]$, then $\sigma _{(\alpha (\log
),\beta ,\gamma )}[f_{1}\pm f_{2}]=\max \{\sigma _{(\alpha (\log ),\beta
,\gamma )}[f_{1}],$ $\sigma _{(\alpha (\log ),\beta ,\gamma )}[f_{2}]\};$%
\newline
(iv) If $\sigma _{(\alpha (\log ),\beta ,\gamma )}[f_{1}]\neq \sigma
_{(\alpha (\log ),\beta ,\gamma )}[f_{2}]$, then $\sigma _{(\alpha (\log
),\beta ,\gamma )}[f_{2}\cdot f_{2}]$ $=$ $\max $ $\{\sigma _{(\alpha (\log
),\beta ,\gamma )}[f_{1}],$ $\sigma _{(\alpha (\log ),\beta ,\gamma
)}[f_{2}]\}$.
\end{proposition}

\qquad By using the properties $T(r,f)=T(r,\frac{1}{f})+O(1)$ and $%
T(r,af)=T(r,f)+O(1)$, $a\in 
%TCIMACRO{\U{2102} }%
%BeginExpansion
\mathbb{C}
%EndExpansion
\setminus \{0\}$, one can obtain the following result.

\begin{proposition}
\label{p4} Let $f$ be a nonconstant meromorphic function with finite $%
(\alpha ,\beta ,\gamma )$-order. Then\newline
(i) $\sigma _{(\alpha ,\beta ,\gamma )}[\frac{1}{f}]=\sigma _{(\alpha ,\beta
,\gamma )}[f]$ $\left( f\not\equiv 0\right) ;$\newline
(ii) $\sigma _{(\alpha (\log ),\beta ,\gamma )}[\frac{1}{f}]=\sigma
_{(\alpha (\log ),\beta ,\gamma )}[f]$ $\left( f\not\equiv 0\right) ;$%
\newline
(iii) If $a\in 
%TCIMACRO{\U{2102} }%
%BeginExpansion
\mathbb{C}
%EndExpansion
\setminus \{0\}$, then $\sigma _{(\alpha ,\beta ,\gamma )}[af]=\sigma
_{(\alpha ,\beta ,\gamma )}[f]$ and $\tau _{(\alpha ,\beta ,\gamma
)}[af]=\tau _{(\alpha ,\beta ,\gamma )}[f];$\newline
(iii) If $a\in 
%TCIMACRO{\U{2102} }%
%BeginExpansion
\mathbb{C}
%EndExpansion
\setminus \{0\}$, then $\sigma _{(\alpha (\log ),\beta ,\gamma )}[af]=\sigma
_{(\alpha (\log ),\beta ,\gamma )}[f]$ and $\tau _{(\alpha (\log ),\beta
,\gamma )}[af]=\tau _{(\alpha (\log ),\beta ,\gamma )}[f]$.\newline
\end{proposition}

\begin{proposition}
\label{p5} Let $f_{1},$ $f_{2}$ be two nonconstant meromorphic functions.
Then the following statements hold:\newline
(i) If $0<\sigma _{(\alpha ,\beta ,\gamma )}[f_{1}]<\sigma _{(\alpha ,\beta
,\gamma )}[f_{2}]<+\infty $ and $\tau _{(\alpha ,\beta ,\gamma
)}[f_{1}]<\tau _{(\alpha ,\beta ,\gamma )}[f_{2}]$, then $\tau _{(\alpha
,\beta ,\gamma )}[f_{1}\pm f_{2}]=\tau _{(\alpha ,\beta ,\gamma
)}[f_{1}\cdot f_{2}]=\tau _{(\alpha ,\beta ,\gamma )}[f_{2}];$\newline
(ii) If $0<\sigma _{(\alpha ,\beta ,\gamma )}[f_{1}]=\sigma _{(\alpha ,\beta
,\gamma )}[f_{2}]=\sigma _{(\alpha ,\beta ,\gamma )}[f_{1}\pm f_{2}]<+\infty 
$, then $\tau _{(\alpha ,\beta ,\gamma )}[f_{1}\pm f_{2}]\leq \max \{\tau
_{(\alpha ,\beta ,\gamma )}[f_{1}],$ $\tau _{(\alpha ,\beta ,\gamma
)}[f_{2}]\}$. Moreover if $\tau _{(\alpha ,\beta ,\gamma )}[f_{1}]\neq \tau
_{(\alpha ,\beta ,\gamma )}[f_{2}]$, then $\tau _{(\alpha ,\beta ,\gamma
)}[f_{1}\pm f_{2}]=\max \{\tau _{(\alpha ,\beta ,\gamma )}[f_{1}],$ $\tau
_{(\alpha ,\beta ,\gamma )}[f_{2}]\};$\newline
(iii) If $0<\sigma _{(\alpha ,\beta ,\gamma )}[f_{1}]=\sigma _{(\alpha
,\beta ,\gamma )}[f_{2}]=\sigma _{(\alpha ,\beta ,\gamma )}[f_{1}\cdot
f_{2}]<+\infty $, then $\tau _{(\alpha ,\beta ,\gamma )}[f_{1}\cdot
f_{2}]\leq \max \{\tau _{(\alpha ,\beta ,\gamma )}[f_{1}],$ $\tau _{(\alpha
,\beta ,\gamma )}[f_{2}]\}$. Moreover if $\tau _{(\alpha ,\beta ,\gamma
)}[f_{1}]\neq \tau _{(\alpha ,\beta ,\gamma )}[f_{2}]$, then $\tau _{(\alpha
,\beta ,\gamma )}[f_{1}\cdot f_{2}]=\max \{\tau _{(\alpha ,\beta ,\gamma
)}[f_{1}],$ $\tau _{(\alpha ,\beta ,\gamma )}[f_{2}]\}$.
\end{proposition}

\begin{proof}
(i) The definition of the $(\alpha ,\beta ,\gamma )$-type implies that for
any given $\varepsilon >0$, there exists a sequence $\{r_{n}$, $n\geq 1\}$
tending to infinity such that%
\begin{equation*}
T(r_{n},f_{2})\geq \exp \left\{ \alpha ^{-1}(\log ((\tau _{(\alpha ,\beta
,\gamma )}[f_{2}]-\varepsilon )(\exp (\beta (\log \gamma (r_{n}))))^{\sigma
_{(\alpha ,\beta ,\gamma )}[f_{2}]}))\right\}
\end{equation*}%
and for sufficiently large values of $r$%
\begin{equation*}
T(r,f_{1})\leq \exp \left\{ \alpha ^{-1}(\log ((\tau _{(\alpha ,\beta
,\gamma )}[f_{1}]+\varepsilon )(\exp \left( \beta \left( \log \gamma \left(
r\right) \right) \right) )^{\sigma _{(\alpha ,\beta ,\gamma
)}[f_{1}]}))\right\} \text{.}
\end{equation*}%
We know that $T(r,f_{1}\pm f_{2})\geq T(r,f_{2})-T(r,f_{1})-\log 2$, we
obtain from above two inequalities that%
\begin{eqnarray*}
T(r_{n},f_{1}\pm f_{2}) &\geq &\exp \left\{ \alpha ^{-1}(\log ((\tau
_{(\alpha ,\beta ,\gamma )}[f_{2}]-\varepsilon )(\exp (\beta (\log \gamma
(r_{n}))))^{\sigma _{(\alpha ,\beta ,\gamma )}[f_{2}]}))\right\} \\
&&-\exp \left\{ \alpha ^{-1}(\log ((\tau _{(\alpha ,\beta ,\gamma
)}[f_{1}]+\varepsilon )(\exp (\beta (\log \gamma (r_{n}))))^{\sigma
_{(\alpha ,\beta ,\gamma )}[f_{1}]}))\right\} -\log 2 \\
&\geq &\exp \left\{ \alpha ^{-1}(\log ((\tau _{(\alpha ,\beta ,\gamma
)}[f_{2}]-2\varepsilon )(\exp (\beta (\log \gamma (r_{n}))))^{\sigma
_{(\alpha ,\beta ,\gamma )}[f_{2}]}))\right\}
\end{eqnarray*}%
provided $\varepsilon $ such that $0<2\varepsilon <\tau _{(\alpha ,\beta
,\gamma )}[f_{2}]-\tau _{(\alpha ,\beta ,\gamma )}[f_{1}]$. It follows from
Proposition \ref{p2} that $\sigma _{(\alpha ,\beta ,\gamma )}[f_{1}\pm
f_{2}]=\sigma _{(\alpha ,\beta ,\gamma )}[f_{2}]$, and by the monotoncity of 
$\alpha $ and above, we obtain that%
\begin{equation*}
\frac{\exp (\alpha (\log T(r_{n},f_{1}\pm f_{2})))}{\left( \exp \left( \beta
\left( \log \gamma \left( r_{n}\right) \right) \right) \right) ^{\sigma
_{(\alpha ,\beta ,\gamma )}[f_{1}\pm f_{2}]}}\geq \tau _{(\alpha ,\beta
,\gamma )}[f_{2}]-2\varepsilon \text{,}
\end{equation*}%
since $\varepsilon $ can be arbitrarily chosen such that $0<2\varepsilon
<\tau _{(\alpha ,\beta ,\gamma )}[f_{2}]-\tau _{(\alpha ,\beta ,\gamma
)}[f_{1}]$, thus%
\begin{equation}
\tau _{(\alpha ,\beta ,\gamma )}[f_{1}\pm f_{2}]\geq \tau _{(\alpha ,\beta
,\gamma )}[f_{2}]\text{.}  \label{a}
\end{equation}%
It remains to prove the converse inequality. Indeed, by applying (\ref{a})
and since%
\begin{equation*}
\sigma _{(\alpha ,\beta ,\gamma )}[f_{1}\pm f_{2}]=\sigma _{(\alpha ,\beta
,\gamma )}[f_{2}]>\sigma _{(\alpha ,\beta ,\gamma )}[f_{1}]=\sigma _{(\alpha
,\beta ,\gamma )}[-f_{1}]
\end{equation*}%
then we get%
\begin{equation}
\tau _{(\alpha ,\beta ,\gamma )}[f_{2}]=\tau _{(\alpha ,\beta ,\gamma
)}[f_{1}\pm f_{2}-f_{1}]\geq \tau _{(\alpha ,\beta ,\gamma )}[f_{1}\pm f_{2}]%
\text{.}  \label{b}
\end{equation}%
We deduce from (\ref{a}) and (\ref{b}) that $\tau _{(\alpha ,\beta ,\gamma
)}[f_{1}\pm f_{2}]=\tau _{(\alpha ,\beta ,\gamma )}[f_{2}]$.

\qquad Now, we prove that $\tau _{(\alpha ,\beta ,\gamma )}[f_{1}\cdot
f_{2}]=\tau _{(\alpha ,\beta ,\gamma )}[f_{2}]$. By the property $%
T(r,f_{1}\cdot f_{2})\geq T(r,f_{1})-T(r,f_{2})+O(1)$ and a similar
discussion as in the above proof, one can easily show that%
\begin{equation}
\tau _{(\alpha ,\beta ,\gamma )}[f_{1}\cdot f_{2}]\geq \tau _{(\alpha ,\beta
,\gamma )}[f_{2}]\text{.}  \label{c}
\end{equation}%
Since $\sigma _{(\alpha ,\beta ,\gamma )}[f_{1}\cdot f_{2}]=\sigma _{(\alpha
,\beta ,\gamma )}[f_{2}]>\sigma _{(\alpha ,\beta ,\gamma )}[f_{1}]=\sigma
_{(\alpha ,\beta ,\gamma )}\left( \frac{1}{f_{1}}\right) $, then by (\ref{c}%
), we get that%
\begin{equation*}
\tau _{(\alpha ,\beta ,\gamma )}[f_{2}]=\tau _{(\alpha ,\beta ,\gamma
)}\left( f_{1}\cdot f_{2}\frac{1}{f_{1}}\right) \geq \tau _{(\alpha ,\beta
,\gamma )}[f_{1}\cdot f_{2}]
\end{equation*}%
and therefore $\tau _{(\alpha ,\beta ,\gamma )}[f_{1}\cdot f_{2}]=\tau
_{(\alpha ,\beta ,\gamma )}[f_{2}]$.

(ii) The definition of the $(\alpha ,\beta ,\gamma )$-type implies that for
any given $\varepsilon >0$ and for all $r$ sufficiently large, we have%
\begin{equation}
T(r,f_{i})\leq \exp \left\{ \alpha ^{-1}(\log ((\tau _{(\alpha ,\beta
,\gamma )}[f_{i}]+\varepsilon )(\exp \left( \beta \left( \log \gamma \left(
r\right) \right) \right) )^{\sigma _{(\alpha ,\beta ,\gamma
)}[f_{i}]}))\right\} \text{, }i=1,2\text{.}  \label{d}
\end{equation}%
By the assumption $0<\sigma _{(\alpha ,\beta ,\gamma )}[f_{1}]=\sigma
_{(\alpha ,\beta ,\gamma )}[f_{2}]=\sigma _{(\alpha ,\beta ,\gamma
)}[f_{1}\pm f_{2}]<+\infty $, we get that%
\begin{eqnarray*}
T(r,f_{1}\pm f_{2}) &\leq &T(r,f_{1})+T(r,f_{2})+\log 2 \\
&\leq &\exp \left\{ \alpha ^{-1}(\log ((\tau _{(\alpha ,\beta ,\gamma
)}[f_{1}]+\varepsilon )(\exp \left( \beta \left( \log \gamma \left( r\right)
\right) \right) )^{\sigma _{(\alpha ,\beta ,\gamma )}[f_{1}]}))\right\} \\
&&+\exp \left\{ \alpha ^{-1}(\log ((\tau _{(\alpha ,\beta ,\gamma
)}[f_{2}]+\varepsilon )(\exp \left( \beta \left( \log \gamma \left( r\right)
\right) \right) )^{\sigma _{(\alpha ,\beta ,\gamma )}[f_{2}]}))\right\}
+\log 2
\end{eqnarray*}%
\begin{equation*}
\leq \exp \left\{ \alpha ^{-1}(\log ((\max (\tau _{(\alpha ,\beta ,\gamma
)}[f_{1}],\tau _{(\alpha ,\beta ,\gamma )}[f_{2}])+3\varepsilon )(\exp
\left( \beta \left( \log \gamma \left( r\right) \right) \right) )^{\sigma
_{(\alpha ,\beta ,\gamma )}[f_{1}\pm f_{2}]}))\right\} \text{.}
\end{equation*}%
By the monotonicity of $\alpha $ and above, we obtain that%
\begin{equation*}
\frac{\exp (\alpha (\log T(r,f_{1}\pm f_{2})))}{\left( \exp \left( \beta
\left( \log \gamma \left( r\right) \right) \right) \right) ^{\sigma
_{(\alpha ,\beta ,\gamma )}[f_{1}\pm f_{2}]}}\leq \max (\tau _{(\alpha
,\beta ,\gamma )}[f_{1}],\tau _{(\alpha ,\beta ,\gamma
)}[f_{2}])+3\varepsilon \text{.}
\end{equation*}%
Since $\varepsilon >0$ can be chosen arbitrarily, then we get that%
\begin{equation}
\tau _{(\alpha ,\beta ,\gamma )}[f_{1}\pm f_{2}]\leq \max \{\tau _{(\alpha
,\beta ,\gamma )}[f_{1}],\tau _{(\alpha ,\beta ,\gamma )}[f_{2}]\}\text{.}
\label{e}
\end{equation}%
Without loss of generality, we may suppose $\tau _{(\alpha ,\beta ,\gamma
)}[f_{1}]<\tau _{(\alpha ,\beta ,\gamma )}[f_{2}]$, then by (\ref{e}) and
since $\sigma _{(\alpha ,\beta ,\gamma )}[f_{1}\pm f_{2}]=\sigma _{(\alpha
,\beta ,\gamma )}[f_{1}]=\sigma _{(\alpha ,\beta ,\gamma )}[-f_{1}]$, it
follows that%
\begin{eqnarray}
\tau _{(\alpha ,\beta ,\gamma )}[f_{2}] &=&\tau _{(\alpha ,\beta ,\gamma
)}[f_{1}\pm f_{2}-f_{1}]  \notag \\
&\leq &\max \{\tau _{(\alpha ,\beta ,\gamma )}[f_{1}\pm f_{2}],\tau
_{(\alpha ,\beta ,\gamma )}[f_{1}]\}=\tau _{(\alpha ,\beta ,\gamma
)}[f_{1}\pm f_{2}]\text{.}  \label{f}
\end{eqnarray}%
We deduce from (\ref{e}) and (\ref{f}) that 
\begin{equation*}
\tau _{(\alpha ,\beta ,\gamma )}[f_{1}\pm f_{2}]=\max \{\tau _{(\alpha
,\beta ,\gamma )}[f_{1}],\tau _{(\alpha ,\beta ,\gamma )}[f_{2}]\}\text{.}
\end{equation*}%
(iii) By a similar discussion as in the above proof and the fact that $%
T(r,f_{1}\cdot f_{2})\leq T(r,f_{2})+T(r,f_{1})$, one can prove that 
\begin{equation}
\tau _{(\alpha ,\beta ,\gamma )}[f_{1}\cdot f_{2}]\leq \max \{\tau _{(\alpha
,\beta ,\gamma )}[f_{1}],\tau _{(\alpha ,\beta ,\gamma )}[f_{2}]\}\text{.}
\label{g}
\end{equation}%
On the other hand, if we suppose that $\tau _{(\alpha ,\beta ,\gamma
)}[f_{1}]<\tau _{(\alpha ,\beta ,\gamma )}[f_{2}]$, then by (\ref{g}) and
since $\sigma _{(\alpha ,\beta ,\gamma )}[f_{1}\cdot f_{2}]=\sigma _{(\alpha
,\beta ,\gamma )}[f_{1}]=\sigma _{(\alpha ,\beta ,\gamma )}\left( \frac{1}{%
f_{1}}\right) $, we get 
\begin{equation}
\tau _{(\alpha ,\beta ,\gamma )}[f_{2}]=\tau _{(\alpha ,\beta ,\gamma
)}\left( f_{1}\cdot f_{2}\frac{1}{f_{1}}\right) \leq \max \{\tau _{(\alpha
,\beta ,\gamma )}[f_{1}\cdot f_{2}],\tau _{(\alpha ,\beta ,\gamma
)}[f_{1}]\}=\tau _{(\alpha ,\beta ,\gamma )}[f_{1}\cdot f_{2}]\text{.}
\label{h}
\end{equation}%
It follows from (\ref{g}) and (\ref{h}) that $\tau _{(\alpha ,\beta ,\gamma
)}[f_{1}\cdot f_{2}]=\max \{\tau _{(\alpha ,\beta ,\gamma )}[f_{1}],\tau
_{(\alpha ,\beta ,\gamma )}[f_{2}]\}$.
\end{proof}

\begin{corollary}
\label{c1}Let $f_{1},$ $f_{2}$ be two nonconstant meromorphic functions.
Then the following statements hold:\newline
(i) If $0<\sigma _{(\alpha ,\beta ,\gamma )}[f_{1}]=\sigma _{(\alpha ,\beta
,\gamma )}[f_{2}]=\sigma _{(\alpha ,\beta ,\gamma )}[f_{1}\pm f_{2}]<+\infty 
$, then $\tau _{(\alpha ,\beta ,\gamma )}[f_{1}]\leq \max \{\tau _{(\alpha
,\beta ,\gamma )}[f_{1}\pm f_{2}],$ $\tau _{(\alpha ,\beta ,\gamma
)}[f_{2}]\};$\newline
(ii) If $0<\sigma _{(\alpha ,\beta ,\gamma )}[f_{1}]=\sigma _{(\alpha ,\beta
,\gamma )}[f_{2}]=\sigma _{(\alpha ,\beta ,\gamma )}[f_{1}\cdot
f_{2}]<+\infty $, then $\tau _{(\alpha ,\beta ,\gamma )}[f_{1}]\leq \max
\{\tau _{(\alpha ,\beta ,\gamma )}[f_{1}\cdot f_{2}],$ $\tau _{(\alpha
,\beta ,\gamma )}[f_{2}]\}$.
\end{corollary}

\begin{proof}
The proof follow immediately from Proposition \ref{p5}. Indeed, since $%
\sigma _{(\alpha ,\beta ,\gamma )}[f_{1}\pm f_{2}]=\sigma _{(\alpha ,\beta
,\gamma )}[f_{2}]=\sigma _{(\alpha ,\beta ,\gamma )}[-f_{2}]$, then%
\begin{equation*}
\tau _{(\alpha ,\beta ,\gamma )}[f_{1}]=\tau _{(\alpha ,\beta ,\gamma
)}[f_{2}\pm f_{1}-f_{2}]\leq \max \{\tau _{(\alpha ,\beta ,\gamma
)}[f_{1}\pm f_{2}],\tau _{(\alpha ,\beta ,\gamma )}[f_{2}]\}\text{.}
\end{equation*}%
Similarly, since $\sigma _{(\alpha ,\beta ,\gamma )}[f_{1}\cdot
f_{2}]=\sigma _{(\alpha ,\beta ,\gamma )}[f_{2}]=\sigma _{(\alpha ,\beta
,\gamma )}\left( \frac{1}{f_{2}}\right) $, then%
\begin{equation*}
\tau _{(\alpha ,\beta ,\gamma )}[f_{1}]=\tau _{(\alpha ,\beta ,\gamma
)}\left( f_{1}\cdot f_{2}\frac{1}{f_{2}}\right) \leq \max \{\tau _{(\alpha
,\beta ,\gamma )}[f_{1}\cdot f_{2}],\tau _{(\alpha ,\beta ,\gamma )}[f_{2}]\}%
\text{.}
\end{equation*}
\end{proof}

\begin{proposition}
\label{p6} Let $f_{1},$ $f_{2}$ be two nonconstant meromorphic functions.
Then the following statements hold:\newline
(i) If $0<\sigma _{(\alpha (\log ),\beta ,\gamma )}[f_{1}]<\sigma _{(\alpha
(\log ),\beta ,\gamma )}[f_{2}]<+\infty $ and $\tau _{(\alpha (\log ),\beta
,\gamma )}[f_{1}]<\tau _{(\alpha (\log ),\beta ,\gamma )}[f_{2}]$, then $%
\tau _{(\alpha (\log ),\beta ,\gamma )}[f_{1}\pm f_{2}]=\tau _{\alpha (\log
),\beta ,\gamma )}[f_{1}\cdot f_{2}]=\tau _{(\alpha (\log ),\beta ,\gamma
)}[f_{2}];$\newline
(ii) If $0<\sigma _{(\alpha (\log ),\beta ,\gamma )}[f_{1}]=\sigma _{(\alpha
(\log ),\beta ,\gamma )}[f_{2}]=\sigma _{(\alpha (\log ),\beta ,\gamma
)}[f_{1}\pm f_{2}]<+\infty $, then $\tau _{(\alpha (\log ),\beta ,\gamma
)}[f_{1}\pm f_{2}]\leq \max \{\tau _{(\alpha (\log ),\beta ,\gamma
)}[f_{1}], $ $\tau _{(\alpha (\log ),\beta ,\gamma )}[f_{2}]\}$. Moreover if 
$\tau _{(\alpha (\log ),\beta ,\gamma )}[f_{1}]\neq \tau _{(\alpha (\log
),\beta ,\gamma )}[f_{2}]$, then $\tau _{(\alpha (\log ),\beta ,\gamma
)}[f_{1}\pm f_{2}]=\max \{\tau _{(\alpha (\log ),\beta ,\gamma )}[f_{1}],$ $%
\tau _{(\alpha (\log ),\beta ,\gamma )}[f_{2}]\};$\newline
(iii) If $0<\sigma _{(\alpha (\log ),\beta ,\gamma )}[f_{1}]=\sigma
_{(\alpha (\log ),\beta ,\gamma )}[f_{2}]=\sigma _{(\alpha (\log ),\beta
,\gamma )}[f_{1}\cdot f_{2}]<+\infty $, then $\tau _{(\alpha (\log ),\beta
,\gamma )}[f_{1}\cdot f_{2}]\leq \max \{\tau _{(\alpha (\log ),\beta ,\gamma
)}[f_{1}],$ $\tau _{(\alpha (\log ),\beta ,\gamma )}[f_{2}]\}$. Moreover if $%
\tau _{(\alpha (\log ),\beta ,\gamma )}[f_{1}]\neq \tau _{(\alpha (\log
),\beta ,\gamma )}[f_{2}]$, then $\tau _{(\alpha (\log ),\beta ,\gamma
)}[f_{1}\cdot f_{2}]=\max \{\tau _{(\alpha (\log ),\beta ,\gamma )}[f_{1}],$ 
$\tau _{(\alpha (\log ),\beta ,\gamma )}[f_{2}]\}$.
\end{proposition}

\begin{corollary}
\label{c2}Let $f_{1},$ $f_{2}$ be two nonconstant meromorphic functions.
Then the following statements hold:\newline
(i) If $0<\sigma _{(\alpha (\log ),\beta ,\gamma )}[f_{1}]=\sigma _{(\alpha
(\log ),\beta ,\gamma )}[f_{2}]=\sigma _{(\alpha (\log ),\beta ,\gamma
)}[f_{1}\pm f_{2}]<+\infty $, then $\tau _{(\alpha (\log ),\beta ,\gamma
)}[f_{1}]\leq \max \{\tau _{(\alpha (\log ),\beta ,\gamma )}[f_{1}\pm
f_{2}], $ $\tau _{(\alpha (\log ),\beta ,\gamma )}[f_{2}]\};$\newline
(ii) If $0<\sigma _{(\alpha (\log ),\beta ,\gamma )}[f_{1}]=\sigma _{(\alpha
(\log ),\beta ,\gamma )}[f_{2}]=\sigma _{(\alpha (\log ),\beta ,\gamma
)}[f_{1}\cdot f_{2}]<+\infty $, then $\tau _{(\alpha (\log ),\beta ,\gamma
)}[f_{1}]\leq \max \{\tau _{(\alpha (\log ),\beta ,\gamma )}[f_{1}\cdot
f_{2}],$ $\tau _{(\alpha (\log ),\beta ,\gamma )}[f_{2}]\}$.
\end{corollary}

\qquad The proofs of Proposition \ref{p6} and Corollary \ref{c2} would run
parallel to that of Proposition \ref{p5} and Corollary \ref{c1}
respectively. We omit the details.

\section{\textbf{Main Results}}

\qquad In this section we present out main results which considerably extend
the results of Kinnunen \cite{13}, Long et al. \cite{L} as well as Bela\"{\i}%
di \cite{b3}, \cite{b5}. Moreover, the results obtained give an answer to
the problem of Chyzhykov and Semochko (\cite{CS}, Remark 1.11).

\begin{theorem}
\label{t1}Let $A_{0}(z),$ $A_{1}(z),...,A_{k-1}(z)$ be entire functions.
Then all nontrivial solutions $f$ of $\left( \ref{1}\right) $ satisfy%
\begin{equation*}
\sup \{\sigma _{(\alpha (\log ),\beta ,\gamma )}[f]\shortmid L(f)=0\}=\sup
\{\sigma _{(\alpha ,\beta ,\gamma )}[A_{j}]\shortmid j=0,...,k-1\}\text{.}
\end{equation*}
\end{theorem}

\begin{theorem}
\label{t2}Let $A_{0}(z),$ $A_{1}(z),...,A_{k-1}(z)$ be entire functions, $%
m=\max \{j\shortmid \sigma _{(\alpha ,\beta ,\gamma )}[A_{j}]\geq \lambda ,$ 
$j=0,...,k-1\}$. Then $\left( \ref{1}\right) $ processes at most $m$ entire
linearly independent solutions $f$ with $\sigma _{(\alpha (\log ),\beta
,\gamma )}[f]<\lambda $.
\end{theorem}

\begin{theorem}
\label{t3}Let $A_{0}(z),$ $A_{1}(z),...,A_{k-1}(z)$ be entire functions such
that $\sigma _{(\alpha ,\beta ,\gamma )}[A_{0}]>\max \{\sigma _{(\alpha
,\beta ,\gamma )}[A_{j}],$ $j=1,...,k-1\}$. Then every solution $%
f(z)\not\equiv 0$ of $\left( \ref{1}\right) $ satisfies $\sigma _{(\alpha
(\log ),\beta ,\gamma )}[f]=\sigma _{(\alpha ,\beta ,\gamma )}[A_{0}]$.
\end{theorem}

\begin{theorem}
\label{t4}Let $A_{0}(z),$ $A_{1}(z),...,A_{k-1}(z)$ be entire functions.
Assume that%
\begin{equation*}
\max \{\sigma _{(\alpha ,\beta ,\gamma )}[A_{j}],j=1,...,k-1\}\leq \sigma
_{(\alpha ,\beta ,\gamma )}[A_{0}]=\sigma _{0}<+\infty
\end{equation*}%
and%
\begin{equation*}
\max \{\tau _{(\alpha ,\beta ,\gamma ),M}[A_{j}]:\sigma _{(\alpha ,\beta
,\gamma )}[A_{j}]=\sigma _{(\alpha ,\beta ,\gamma )}[A_{0}]>0\}<\tau
_{(\alpha ,\beta ,\gamma ),M}[A_{0}]=\tau _{M}\text{.}
\end{equation*}%
Then every solution $f(z)\not\equiv 0$ of $\left( \ref{1}\right) $ satisfies 
$\sigma _{(\alpha (\log ),\beta ,\gamma )}[f]=\sigma _{(\alpha ,\beta
,\gamma )}[A_{0}].$
\end{theorem}

\qquad By combining Theorem \ref{t3} and Theorem \ref{t4}, we obtain the
following result.

\begin{corollary}
\label{c3}Let $A_{0}(z),$ $A_{1}(z),...,A_{k-1}(z)$ be entire functions.
Assume that%
\begin{equation*}
\max \{\sigma _{(\alpha ,\beta ,\gamma )}[A_{j}],j=1,...,k-1\}<\sigma
_{(\alpha ,\beta ,\gamma )}[A_{0}]=\sigma _{0}<+\infty \text{,}
\end{equation*}%
or%
\begin{equation*}
\max \{\sigma _{(\alpha ,\beta ,\gamma )}[A_{j}],j=1,...,k-1\}\leq \sigma
_{(\alpha ,\beta ,\gamma )}[A_{0}]=\sigma _{0}<+\infty \text{ }(0<\sigma
_{0}<+\infty )
\end{equation*}%
and%
\begin{equation*}
\max \{\tau _{(\alpha ,\beta ,\gamma ),M}[A_{j}]:\sigma _{(\alpha ,\beta
,\gamma )}[A_{j}]=\sigma _{(\alpha ,\beta ,\gamma )}[A_{0}]>0\}<\tau
_{(\alpha ,\beta ,\gamma ),M}[A_{0}]=\tau _{M}\text{ }(0<\tau _{M}<+\infty )%
\text{.}
\end{equation*}%
Then every solution $f(z)\not\equiv 0$ of $\left( \ref{1}\right) $ satisfies 
$\sigma _{(\alpha (\log ),\beta ,\gamma )}[f]=\sigma _{(\alpha ,\beta
,\gamma )}[A_{0}]$.

\begin{remark}
Theorems \ref{t1}-\ref{t3}\textbf{\ }are\textbf{\ }counterparts of Theorems
1.8--1.10 of Chyzhykov and Semochko \cite{CS} for the $\left( \alpha ,\beta
,\gamma \right) $-order.
\end{remark}
\end{corollary}

\section{\textbf{Some Lemmas}}

\qquad In this section we present some lemmas which will be needed in the
sequel.

\begin{lemma}
\label{l1} Let $f$ be a meromorphic function of order $\sigma _{(\alpha
(\log ),\beta ,\gamma )}[f]$ $=\sigma $, $k\in \mathbb{N}$. Then, for any
given $\varepsilon >0$,%
\begin{equation*}
m\left( r,\frac{f^{(k)}}{f}\right) =O(\exp \left\{ \alpha ^{-1}((\sigma
+\varepsilon )\beta \left( \log \gamma \left( r\right) \right) )\right\} )%
\text{,}
\end{equation*}%
outside, possibly, an exceptional set $F\subset \left[ 0,+\infty \right) $
of finite linear measure.
\end{lemma}

\begin{proof}
Let $k=1$. The definition of $\sigma _{(\alpha (\log ),\beta ,\gamma )}[f]$
implies that for any given $\varepsilon >0$, there exists $r_{0}>1$, such
that for all $r>r_{0}$,%
\begin{equation}
T(r,f)\leq \exp ^{[2]}\left\{ \alpha ^{-1}((\sigma +\varepsilon )\beta
\left( \log \gamma \left( r\right) \right) )\right\} \text{.}  \label{l}
\end{equation}%
It follows from (\ref{l}) and the lemma of logarithmic derivative and the
condition $\alpha (\log ^{[2]}x)=o(\beta (\log \gamma (x)))$ as $%
x\rightarrow +\infty $ that%
\begin{eqnarray}
m\left( r,\frac{f^{\prime }}{f}\right) &=&O(\log T(r,f)+\log r)  \notag \\
&=&O\left( \exp \left\{ \alpha ^{-1}((\sigma +\varepsilon )\beta \left( \log
\gamma \left( r\right) \right) )\right\} \right) ,\text{ }r\notin F,
\label{m}
\end{eqnarray}%
where $F\subset \left[ 0,+\infty \right) $ is of finite linear measure.

\qquad Now, we assume that for some $k\in \mathbb{N}$,%
\begin{equation}
m\left( r,\frac{f^{(k)}}{f}\right) =O\left( \exp \left\{ \alpha
^{-1}((\sigma +\varepsilon )\beta \left( \log \gamma \left( r\right) \right)
)\right\} \right) ,\text{ }r\notin F.  \label{n}
\end{equation}%
Since $N(r,f^{(k)})\leq (k+1)N(r,f)$, we deduce%
\begin{eqnarray}
T(r,f^{(k)}) &=&m(r,f^{(k)})+N(r,f^{(k)})  \notag \\
&\leq &m\left( r,\frac{f^{(k)}}{f}\right) +m(r,f)+(k+1)N(r,f)  \notag \\
&\leq &(k+1)T(r,f)+O\left( \exp \left\{ \alpha ^{-1}((\sigma +\varepsilon
)\beta \left( \log \gamma \left( r\right) \right) )\right\} \right)  \notag
\\
&\leq &O\left( \exp ^{[2]}\left\{ \alpha ^{-1}((\sigma +\varepsilon )\beta
\left( \log \gamma \left( r\right) \right) )\right\} \right) \text{.}
\label{x}
\end{eqnarray}%
It follows from (\ref{m}) and (\ref{x}) that%
\begin{equation*}
m\left( r,\frac{f^{(k+1)}}{f^{(k)}}\right) =m\left( r,\frac{\left(
f^{(k)}\right) ^{\prime }}{f^{(k)}}\right) =O(\log T(r,f^{\left( k\right)
})+\log r)
\end{equation*}%
\begin{equation*}
=O\left( \exp \left\{ \alpha ^{-1}((\sigma +\varepsilon )\beta \left( \log
\gamma \left( r\right) \right) )\right\} \right) ,r\notin F.
\end{equation*}%
Thus%
\begin{eqnarray*}
m\left( r,\frac{f^{(k+1)}}{f}\right) &\leq &m\left( r,\frac{f^{(k+1)}}{%
f^{(k)}}\right) +m\left( r,\frac{f^{(k)}}{f}\right) \\
&=&O\left( \exp \left\{ \alpha ^{-1}((\sigma +\varepsilon )\beta \left( \log
\gamma \left( r\right) \right) )\right\} \right) ,\text{ }r\notin F\text{.}
\end{eqnarray*}
\end{proof}

\begin{lemma}
\label{l2}(\cite{gg,19}) Let $g:[0,+\infty )\rightarrow 
%TCIMACRO{\U{211d} }%
%BeginExpansion
\mathbb{R}
%EndExpansion
$ and $h:[0,+\infty )\rightarrow 
%TCIMACRO{\U{211d} }%
%BeginExpansion
\mathbb{R}
%EndExpansion
$ be monotone nondecreasing functions such that $g(r)\leq h(r)$ outside of
an exceptional set $E$ of finite linear measure or finite logarithmic
measure. Then, for any $d>1$, there exists $r_{0}>0$ such that $g(r)\leq
h(dr)$ for all $r>r_{0}$.
\end{lemma}

\begin{lemma}
\label{l3}Let $f$ be a meromorphic function. Then $\sigma _{(\alpha ,\beta
,\gamma )}[f^{\prime }]=\sigma _{(\alpha ,\beta ,\gamma )}[f]$.
\end{lemma}

\begin{proof}
Set $\sigma _{(\alpha ,\beta ,\gamma )}[f]=\sigma $. From the definition of $%
(\alpha ,\beta ,\gamma )$-order, for any given $\varepsilon >0$, there
exists $r_{0}>1$, such that for all $r\geq r_{0}$,%
\begin{equation*}
\log T(r,f)\leq \alpha ^{-1}((\sigma +\varepsilon )\beta \left( \log \gamma
\left( r\right) \right) )\text{.}
\end{equation*}%
Obviously, $T(r,f^{\prime })\leq 2T(r,f)+m\left( r,\frac{f^{\prime }}{f}%
\right) $. By the lemma of logarithmic derivative (p. 34 in \cite{1}) and
the condition $\alpha (\log ^{[2]}x)=o(\beta (\log \gamma (x)))$ as $%
x\rightarrow +\infty $, we have%
\begin{eqnarray*}
\log T(r,f^{\prime }) &\leq &\log T(r,f)+\log \left\{ O\left( \log r+\log
T\left( r,f\right) \right) \right\} +O\left( 1\right) \\
&\leq &\log T(r,f)+\log \left( \log r\right) +\log \log T\left( r,f\right)
+O\left( 1\right) \\
&\leq &\alpha ^{-1}\left( \left( \sigma +4\varepsilon \right) \beta \left(
\log \gamma \left( r\right) \right) \right) ,\text{ }r\notin F,
\end{eqnarray*}%
where $F\subset \lbrack 0,+\infty )$ is a set of finite linear measure. By
Lemma \ref{l2}, for $d=2$ and all $r>r_{0},$ we have%
\begin{equation*}
\log T(r,f^{\prime })\leq \alpha ^{-1}\left( \left( \sigma +4\varepsilon
\right) \beta \left( \log \gamma \left( 2r\right) \right) \right) .
\end{equation*}%
Therefore, by using $\gamma (2r)\leq 2\gamma (r)$ and $\beta
(x+O(1))=(1+o(1))\beta (x)$ as $x\rightarrow +\infty $ 
\begin{equation*}
\alpha \left( \log T(r,f^{\prime })\right) \leq \left( \sigma +4\varepsilon
\right) \beta \left( \log \gamma \left( 2r\right) \right) \leq \left( \sigma
+4\varepsilon \right) \beta \left( \log \left( 2\gamma \left( r\right)
\right) \right)
\end{equation*}%
\begin{equation*}
=\left( \sigma +4\varepsilon \right) \beta \left( \log 2+\log \gamma \left(
r\right) \right) =\left( \sigma +4\varepsilon \right) (1+o(1))\beta \left(
\log \left( \gamma \left( r\right) \right) \right) .
\end{equation*}%
It is implies that $\sigma _{(\alpha ,\beta ,\gamma )}[f]\geq \sigma
_{(\alpha ,\beta ,\gamma )}[f^{\prime }]$.

\qquad On the other hand, we prove the inequality $\sigma _{(\alpha ,\beta
,\gamma )}[f]\leq \sigma _{(\alpha ,\beta ,\gamma )}[f^{\prime }]$. The
definition of $\sigma _{(\alpha ,\beta ,\gamma )}[f^{\prime }]=\sigma
^{\prime }$ implies that for any given above $\varepsilon >0$ and for all $r$
sufficiently large, we have%
\begin{equation*}
T(r,f^{\prime })\leq \exp \left\{ \alpha ^{-1}((\sigma ^{\prime
}+\varepsilon )\beta \left( \log \gamma \left( r\right) \right) )\right\} 
\text{.}
\end{equation*}%
Since $T(r,f)\leq O(T(2r,f^{\prime })+\log r)$, $r\rightarrow +\infty $ (cf. 
\cite{LY}), then by using $\gamma (2r)\leq 2\gamma (r)$, $\beta
(x+O(1))=(1+o(1))\beta (x)$ and $\alpha (\log ^{[2]}x)=o(\beta (\log \gamma
(x)))$ as $x\rightarrow +\infty ,$ we can get that%
\begin{eqnarray*}
T(r,f) &\leq &O\left( \exp \left\{ \alpha ^{-1}((\sigma ^{\prime
}+\varepsilon )\beta (\log \gamma (2r))\right\} \right) +\log r) \\
&\leq &\exp \left\{ \alpha ^{-1}\left( \left( \sigma ^{\prime }+2\varepsilon
\right) \beta \left( \log \left( \gamma \left( 2r\right) \right) \right)
\right) \right\} \\
\text{ } &\leq &\exp \left\{ \alpha ^{-1}\left( \left( \sigma ^{\prime
}+2\varepsilon \right) \left( 1+o\left( 1\right) \right) \beta \left( \log
\left( \gamma \left( r\right) \right) \right) \right) \right\} ,
\end{eqnarray*}%
\begin{equation*}
i.e.,~\log T(r,f))\leq \alpha ^{-1}((\sigma ^{\prime }+2\varepsilon )\left(
1+o\left( 1\right) \right) \beta \left( \log \gamma \left( r\right) \right)
)),
\end{equation*}%
\begin{equation*}
i.e.,~\alpha \left( \log T\left( r,f\right) \right) \leq (\sigma ^{\prime
}+2\varepsilon )\left( 1+o\left( 1\right) \right) \beta \left( \log \gamma
\left( r\right) \right) ,\text{ }r\rightarrow +\infty \text{.}
\end{equation*}%
Since $\varepsilon >0$ is an arbitrary number, we obtain that $\sigma
_{(\alpha ,\beta ,\gamma )}[f]\leq \sigma ^{\prime }=\sigma _{(\alpha ,\beta
,\gamma )}[f^{\prime }]$ and therefore $\sigma _{(\alpha ,\beta ,\gamma
)}[f^{\prime }]=\sigma _{(\alpha ,\beta ,\gamma )}[f]$. Hence the proof
follows.
\end{proof}

\begin{remark}
\label{r1} In the line of Lemma \ref{l3} one can easily deduce that $\sigma
_{(\alpha (\log ),\beta ,\gamma )}[f^{\prime }]=\sigma _{(\alpha (\log
),\beta ,\gamma )}[f],$ where $f$ is a meromorphic function.
\end{remark}

\begin{lemma}
\label{l4}(\cite{17, gl,v}) Let $f$ be a transcendental entire function, let 
$0<\delta <\frac{1}{4\text{ }}$ and $z$ such that $|z|=r$ and $|f(z)|$ $%
>M(r,f)\nu (r,f)^{-\frac{1}{4}+\delta }$. Then there exists a set $E\subset 
%TCIMACRO{\U{211d} }%
%BeginExpansion
\mathbb{R}
%EndExpansion
_{+}$ of finite logarithmic measure such that%
\begin{equation*}
f^{(m)}(z)=\left( \frac{\nu (r,f)}{z}\right) ^{m}(1+o(1))f(z)
\end{equation*}%
holds for integer $m\geq 0$ and $r\notin E$, where $\nu (r,f)$ is the
central index of $f$.
\end{lemma}

\begin{lemma}
\label{l5} (\cite[p. 10]{19}) Let $P(z)=a_{n}z^{n}+a_{n-1}z^{n-1}+\cdots
+a_{1}z+a_{0}$ be a polynomial, where $a_{n}\neq 0$. Then all zeros of $%
P\left( z\right) $ lie in the disc $D(0,r)$ of radius%
\begin{equation*}
r\leq 1+\underset{0\leq k\leq n-1}{\max }\left( \left\vert \frac{a_{k}}{a_{n}%
}\right\vert \right) \text{.}
\end{equation*}
\end{lemma}

\begin{lemma}
\label{l6}Let $f$ be a meromorphic function with $\sigma _{(\alpha ,\beta
,\gamma )}[f]=\sigma _{0}\in (0,+\infty )$. Then for all $\mu $ $(<\sigma
_{0})$, there exists a set $I\subset \left( 1.+\infty \right) $ of infinite
logarithmic measure such that $\alpha \left( \log T\left( r,f\right) \right)
>\mu \beta \left( \log \gamma \left( r\right) \right) $ holds for all $r\in
I $.
\end{lemma}

\begin{proof}
The definition of $(\alpha ,\beta ,\gamma )$-order implies that there exists
a sequence $(R_{j})_{j=1}^{+\infty }$ satisfying%
\begin{equation*}
\left( 1+\frac{1}{j}\right) R_{j}<R_{J+1},\text{ \ \ \ \ \ }\underset{%
j\rightarrow +\infty }{\lim }\frac{\alpha (\log T(R_{j},f))}{\beta (\log
\gamma (R_{j}))}=\sigma _{0}\text{.}
\end{equation*}%
From the equality above, for any given $\varepsilon \in (0,\sigma _{0}-\mu )$%
, there exists an integer $j_{1}$ such that for $j\geq j_{1}$,%
\begin{equation}
\alpha (\log T(R_{j},f))>(\sigma _{0}-\varepsilon )\beta (\log \gamma
(R_{j}))\text{.}  \label{i}
\end{equation}%
Since $\mu <\sigma _{0}-\varepsilon $, there exists an integer $j_{2}$ such
that for $j\geq j_{2}$,%
\begin{equation*}
\frac{\sigma _{0}-\varepsilon }{\mu }>\frac{\beta \left( \log \gamma \left(
\left( 1+\frac{1}{j}\right) R_{j}\right) \right) }{\beta (\log \gamma
(R_{j}))}\text{.}
\end{equation*}%
It follows from this inequality and (\ref{i}) that for $j\geq j_{3}=\max
\{j_{1},j_{2}\}$ and for any $r\in \left[ R_{j},\left( 1+\frac{1}{j}\right)
R_{j}\right] $,%
\begin{eqnarray*}
\alpha \left( \log T\left( r,f\right) \right) &\geq &\alpha (\log
T(R_{j},f))>(\sigma _{0}-\varepsilon )\beta (\log \gamma (R_{j})) \\
&=&\frac{\sigma _{0}-\varepsilon }{\mu }\mu \frac{\beta (\log \gamma (R_{j}))%
}{\beta \left( \log \gamma \left( r\right) \right) }\beta \left( \log \gamma
\left( r\right) \right) \\
&\geq &\frac{\sigma _{0}-\varepsilon }{\mu }\frac{\beta (\log \gamma (R_{j}))%
}{\beta \left( \log \gamma \left( \left( 1+\frac{1}{j}\right) R_{j}\right)
\right) }\mu \beta \left( \log \gamma \left( r\right) \right) \\
&>&\mu \beta \left( \log \gamma \left( r\right) \right) \text{.}
\end{eqnarray*}%
Set $I=\underset{j=j_{3}}{\overset{+\infty }{\cup }}\left[ R_{j},\left( 1+%
\frac{1}{j}\right) R_{j}\right] $. It is easy to show that $I$ is of
infinite logarithmic measure,%
\begin{equation*}
m_{l}I:=\underset{I}{\int }\frac{dr}{r}=\underset{j=j_{3}}{\overset{+\infty }%
{\sum }}\overset{\left( 1+\frac{1}{j}\right) R_{j}}{\underset{R_{j}}{\int }}%
\frac{dr}{r}=\underset{j=j_{3}}{\overset{+\infty }{\sum }}\log \left( 1+%
\frac{1}{j}\right) =+\infty .
\end{equation*}
\end{proof}

\qquad We can also prove the following result by using similar reason as in
the proof of Lemma \ref{l6}.

\begin{lemma}
\label{l7}Let $f$ be an entire function with $\sigma _{(\alpha ,\beta
,\gamma )}[f]=\sigma _{0}\in (0,+\infty )$ and $\tau _{(\alpha ,\beta
,\gamma ),M}[f]\in (0,+\infty )$. Then for any given $\omega <\tau _{(\alpha
,\beta ,\gamma ),M}[f]$, there exists a set $I\subset \left( 1.+\infty
\right) $ of infinite logarithmic measure such that for all $r\in I$,%
\begin{equation*}
\exp \left\{ \alpha (\log ^{[2]}M(r,f))\right\} >\omega \left( \exp \left\{
\beta \left( \log \gamma \left( r\right) \right) \right\} \right) ^{\sigma
_{0}}\text{.}
\end{equation*}
\end{lemma}

\begin{lemma}
\label{l8}(\cite{20}) Let $f$ be a solution of $\left( \ref{1}\right) $, and
let $1\leq p<+\infty $. Then for all $0<r<R$, where $0<R<+\infty $,%
\begin{equation*}
m_{p}(r,f)^{p}\leq C\left( \underset{j=0}{\overset{k-1}{\sum }}\text{ }%
\underset{0}{\overset{2\pi }{\int }}\text{ }\underset{0}{\overset{r}{\int }}%
\left\vert A_{j}(se^{i\theta })\right\vert ^{\frac{p}{k-j}}dsd\theta
+1\right) \text{,}
\end{equation*}%
where $C>0$ is a constant which depends on $p$ and the initial value of $f$
in a point $z_{0}$, where $A_{j}\neq 0$ for some $j=0,...,k-1$, and where%
\begin{equation*}
m_{p}(r,f)^{p}\leq \frac{1}{2\pi }\underset{0}{\overset{2\pi }{\int }}%
(\left\vert \log ^{+}\left\vert f(re^{i\theta })\right\vert \right\vert
)^{p}d\theta \text{.}
\end{equation*}
\end{lemma}

\qquad The following logarithmic derivative estimation was found in \cite{gs}
from Gundersen.

\begin{lemma}
\label{l9}(\cite{gs}) Let $f$ be a transcendental meromorphic function, and
let $\xi >1$ be a given constant. Then there exist a set $E\subset \left(
1.+\infty \right) $ with finite logarithmic measure and a constant $B>0$
that depends only on $\xi $, and $i$, $j$, $0\leq i<j\leq k-1$, such that
for all $z$ satisfying $\left\vert z\right\vert =r\notin \lbrack 0,1]\cup E$,%
\begin{equation*}
\left\vert \frac{f^{(j)}(z)}{f^{\left( i\right) }(z)}\right\vert \leq
B\left\{ \frac{T(\xi r,f)}{r}(\log ^{\xi }r)\log T(\xi r,f)\right\} ^{j-i}%
\text{.}
\end{equation*}
\end{lemma}

\begin{lemma}
\label{l10} Let $A_{0}(z),$ $A_{1}(z),...,A_{k-1}(z)$ be entire functions.
Then every nontrivial solution $f$ of $\left( \ref{1}\right) $ satisfies%
\begin{equation*}
\sigma _{(\alpha (\log ),\beta ,\gamma )}[f]\leq \max \{\sigma _{(\alpha
,\beta ,\gamma )}[A_{j}]:j=0,1,...,k-1\}\text{.}
\end{equation*}
\end{lemma}

\begin{proof}
Set%
\begin{equation*}
\varpi =\max \{\sigma _{(\alpha ,\beta ,\gamma )}[A_{j}]:j=0,1,...,k-1\}%
\text{.}
\end{equation*}%
By the definition of $\sigma _{(\alpha ,\beta ,\gamma )}[A_{j}]$, for any
given $\varepsilon >0$ and for sufficiently large $r$,%
\begin{equation}
M(r,A_{j})\leq \exp ^{[2]}\left\{ \alpha ^{-1}((\varpi +\varepsilon )\beta
\left( \log \gamma \left( r\right) \right) )\right\} ,\text{ }j=0,1,...,k-1%
\text{.}  \label{j}
\end{equation}%
By Lemma \ref{l8} for $p=1$ and (\ref{j}), we have%
\begin{equation*}
T(r,f)=m(r,f)\leq 2\pi C\left( 1+\underset{j=0}{\overset{k-1}{\sum }}%
rM(r,A_{j})\right)
\end{equation*}%
\begin{equation*}
\leq 2\pi C\left( 1+kr\exp ^{[2]}\left\{ \alpha ^{-1}((\varpi +\varepsilon
)\beta \left( \log \gamma \left( r\right) \right) )\right\} \right)
\end{equation*}%
\begin{equation}
\leq \exp ^{[2]}\left\{ \alpha ^{-1}((\varpi +2\varepsilon )\beta \left(
\log \gamma \left( r\right) \right) )\right\} \text{.}  \label{k}
\end{equation}%
Therefore, we get from (\ref{k}) and Proposition \ref{p1} that%
\begin{equation*}
\sigma _{(\alpha (\log ),\beta ,\gamma )}[f]\leq \max \{\sigma _{(\alpha
,\beta ,\gamma )}[A_{j}]:j=0,1,...,k-1\}\text{.}
\end{equation*}
\end{proof}

\section{\textbf{Proof of the Main Results}}

\textbf{Proof of Theorem \ref{t1}}. Set $\Upsilon =\sup \{\sigma _{(\alpha
(\log ),\beta ,\gamma )}[f]\mid L(f)=0\}$ and $\digamma =\sup \{\sigma
_{(\alpha ,\beta ,\gamma )}[A_{j}]\mid j=0,...,k-1\}$.

\qquad First we prove that $\digamma \leq \Upsilon $. If $\Upsilon =+\infty $%
, it is trivial. Hence we just consider the case of $\Upsilon <+\infty $.
Let $f_{1}(z),f_{2}(z),...,f_{k}(z)$ be a solution base of (\ref{1}) with $%
\sigma _{(\alpha (\log ),\beta ,\gamma )}[f_{j}]<+\infty $, $j=1,...,k$. It
is clear that $W=W(f_{1},f_{2},...,f_{k})\neq 0$ by the properties of the
Wronsky determinant.

\qquad It follows from Proposition \ref{p2}, Proposition \ref{p4} and Lemma %
\ref{l3} that $\sigma _{(\alpha (\log ),\beta ,\gamma )}[W]<+\infty $. By
properties of the Wronsky determinant (\cite[p. 55]{19}),%
\begin{equation*}
A_{k-s}(z)=-W_{k-s}((f_{1},f_{2},...,f_{k})\cdot W^{-1},\text{ }s\in
\{1,...,k\}\text{,}
\end{equation*}%
where%
\begin{equation*}
W_{j}(f_{1},f_{2},...,f_{k})=\left\vert 
\begin{array}{ccc}
f_{1} & \cdot \cdot \cdot & f_{k} \\ 
\begin{array}{c}
\cdot \\ 
\cdot \\ 
\cdot%
\end{array}
& 
\begin{array}{c}
\cdot \\ 
\cdot \\ 
\cdot%
\end{array}
& 
\begin{array}{c}
\cdot \\ 
\cdot \\ 
\cdot%
\end{array}
\\ 
f_{1}^{(j-1)} & \cdot \cdot \cdot & f_{k}^{(j-1)} \\ 
f_{1}^{(j)} & \cdot \cdot \cdot & f_{k}^{(j)} \\ 
f_{1}^{(j+1)} & \cdot \cdot \cdot & f_{k}^{(j+1)} \\ 
\begin{array}{c}
\cdot \\ 
\cdot \\ 
\cdot%
\end{array}
& 
\begin{array}{c}
\cdot \\ 
\cdot \\ 
\cdot%
\end{array}
& 
\begin{array}{c}
\cdot \\ 
\cdot \\ 
\cdot%
\end{array}
\\ 
f_{1}^{(k-1)} & \cdot \cdot \cdot & f_{k}^{(k-1)}%
\end{array}%
\right\vert \text{.}
\end{equation*}%
In view of Proposition \ref{p2} we can conclude that $\sigma _{(\alpha (\log
),\beta ,\gamma )}[A_{i}]<+\infty $, $i=0,1,...,k-1$. By Lemma \ref{l1} to $%
f_{i}$, $i=1,...,k$%
\begin{equation*}
m\left( r,\frac{f_{i}^{(l)}}{f_{i}}\right) =O(\exp (\alpha ^{-1}((\Upsilon
+\varepsilon )\beta \left( \log \gamma \left( r\right) \right) )))\text{, }%
l=1,2,...,k,\text{ }r\notin F\text{.}
\end{equation*}%
We now apply the standard order reduction procedure (\cite[p. 53-57]{19}).
Let us denote%
\begin{equation*}
\nu _{1}(z):=\frac{d}{dz}\left( \frac{f(z)}{f_{1}(z)}\right) \text{,}
\end{equation*}%
$A_{k}=1$ and $\nu _{1}^{(-1)}:=\frac{f}{f_{1}}$, i.e., $(\nu
_{1}^{(-1)})^{^{\prime }}:=\nu _{1}$. Hence,%
\begin{equation}
f^{(l)}=\underset{m=0}{\overset{l}{\sum }}\binom{l}{m}f_{1}^{(m)}\nu
_{1}^{(l-1-m)}\text{, }l=0,...,k\text{.}  \label{2}
\end{equation}%
Substituting (\ref{2}) into (\ref{1}) and using the fact that $f_{1}$ solves
(\ref{1}), we obtain 
\begin{equation}
\nu _{1}^{(k-1)}+A_{1,k-2}(z)\nu _{1}^{(k-2)}+\cdots +A_{1,0}(z)\nu _{1}=0,
\label{3}
\end{equation}%
where%
\begin{equation*}
A_{1,j}=A_{j+1}+\underset{m=1}{\overset{k-j-1}{\sum }}\binom{j+1+m}{m}%
A_{j+1+m}\frac{f_{1}^{(m)}}{f_{1}},\text{ }j=0,...,k-2\text{.}
\end{equation*}%
By $\Upsilon <+\infty $ and Lemma \ref{l3}, the meromorphic functions%
\begin{equation}
\nu _{1,j}(z)=\frac{d}{dz}\left( \frac{f_{j+1}(z)}{f_{1}(z)}\right) ,\text{ }%
j=1,...,k-1,  \label{4}
\end{equation}%
form a solution base to (\ref{3}) of finite $(\alpha (\log ),\beta ,\gamma )$%
-order.

\qquad Next, we claim that%
\begin{equation}
m(r,A_{i})=O\left( \exp \left\{ \alpha ^{-1}((\Upsilon +\varepsilon )\beta
\left( \log \gamma \left( r\right) \right) )\right\} \right) ,\text{ }%
r\notin F,\text{ }i=0,...,k-1\text{,}  \label{5}
\end{equation}%
when%
\begin{equation}
m(r,A_{1,j})=O\left( \exp \left\{ \alpha ^{-1}((\Upsilon +\varepsilon )\beta
\left( \log \gamma \left( r\right) \right) )\right\} \right) ,\text{ }%
r\notin F,\text{ }j=0,...,k-2\text{.}  \label{6}
\end{equation}%
In fact, we prove it by induction on $i$ following \cite[p. 10]{19}. By (\ref%
{3}), for $j=k-2$, we have $A_{1,k-2}=A_{k-1}+k\frac{f^{\prime }}{f}$. By
Lemma \ref{l1} and (\ref{6}), we get that%
\begin{eqnarray*}
m(r,A_{k-1}) &\leq &m(r,A_{1,k-2})+m\left( r,\frac{f^{\prime }}{f}\right)
+O(1) \\
&=&O\left( \exp \left\{ \alpha ^{-1}((\Upsilon +\varepsilon )\beta \left(
\log \gamma \left( r\right) \right) )\right\} \right) \text{.}
\end{eqnarray*}%
We assume that%
\begin{equation}
m(r,A_{i})=O\left( \exp \left\{ \alpha ^{-1}((\Upsilon +\varepsilon )\beta
\left( \log \gamma \left( r\right) \right) )\right\} \right) ,\text{ }%
i=k-1,...,k-l\text{.}  \label{7}
\end{equation}%
Since%
\begin{equation*}
A_{1,k-(l+2)}=A_{k-(l+1)}+\underset{m=1}{\overset{l+1}{\sum }}\binom{m+k-l-1%
}{m}A_{m+k-l-1}\frac{f_{1}^{(m)}}{f_{1}}\text{,}
\end{equation*}%
by Lemma \ref{l1}, (\ref{6}) and (\ref{7}), we obtain that%
\begin{eqnarray}
m(r,A_{k-(l+1)}) &\leq &m(r,A_{1,k-(l+2)})+m(r,A_{k-1})+\cdots +m(r,A_{k-l})
\notag \\
&&+m\left( r,\frac{f_{1}^{\prime }}{f_{1}}\right) +\cdots +m\left( r,\frac{%
f_{1}^{(l+1)}}{f_{1}}\right) +O(1)  \label{8} \\
&=&O\left( \exp \left\{ \alpha ^{-1}((\Upsilon +\varepsilon )\beta \left(
\log \gamma \left( r\right) \right) )\right\} \right) \text{, }r\notin F%
\text{.}  \notag
\end{eqnarray}%
We may now proceed as above to further reduce the order of (\ref{3}). In
each reduction step, according to (\ref{4}), we obtain a solution base of
meromorphic functions of finite $(\alpha (\log ),\beta ,\gamma )$-order
according to (\ref{4}), and the implication (\ref{6}) to (\ref{5}) remains
valid. Hence, we finally obtain an equation of type $w^{\prime }+B(z)w=0$,
and $w$ is any solution of the equation with $\sigma _{(\alpha (\log ),\beta
,\gamma )}[w]<+\infty $. Then%
\begin{equation*}
m(r,B)=m\left( r,\frac{w^{\prime }}{w}\right) =O\left( \exp \left\{ \alpha
^{-1}((\Upsilon +\varepsilon )\beta \left( \log \gamma \left( r\right)
\right) )\right\} \right) \text{, }r\notin F\text{.}
\end{equation*}%
Observing the reasoning corresponding to (\ref{5}) and (\ref{6}) in the
subsequent reduction steps, we see that%
\begin{equation*}
m(r,A_{j})=O(\exp (\alpha ^{-1}((\Upsilon +\varepsilon )\beta \left( \log
\gamma \left( r\right) \right) ))),\text{ }j=0,...,k-1,\text{ }r\notin F%
\text{.}
\end{equation*}%
It implies that%
\begin{equation*}
T(r,A_{j})=O\left( \exp \left\{ \alpha ^{-1}((\Upsilon +\varepsilon )\beta
\left( \log \gamma \left( r\right) \right) )\right\} \right) ,\text{ }%
j=0,...,k-1,\text{ }r\notin F\text{.}
\end{equation*}%
By using Lemma \ref{l2} for $d=2$, $\gamma (2r)\leq 2\gamma (r)$ and $\beta
(x+O(1))=(1+o(1))\beta (x)$ as $x\rightarrow +\infty $, for sufficiently
large $r$, $j=0,...,k-1$, we obtain that%
\begin{equation*}
T(r,A_{j})=O\left( \exp \left\{ \alpha ^{-1}((\Upsilon +\varepsilon )\beta
(\log \gamma (2r))))\right\} \right)
\end{equation*}%
\begin{equation*}
\leq \exp \left\{ \alpha ^{-1}((\Upsilon +2\varepsilon )\beta \left( \log
2+\log \gamma \left( r\right) \right) ))\right\} =\exp \left\{ \alpha
^{-1}((\Upsilon +2\varepsilon )\left( 1+o\left( 1\right) \right) \beta
\left( \log \gamma \left( r\right) \right) ))\right\} .
\end{equation*}%
Hence $\frac{\alpha (\log T(r,A_{j}))}{\beta \left( \log \gamma \left(
r\right) \right) }\leq \Upsilon +2\varepsilon $. This implies that $\digamma
\leq \Upsilon $.

\qquad We next prove the converse inequality under the assumption that $%
\digamma \leq +\infty $.

\qquad By Lemma \ref{l4}, there exists a set $E\in 
%TCIMACRO{\U{211d} }%
%BeginExpansion
\mathbb{R}
%EndExpansion
_{+}$ of finite logarithmic measure, such that for all $z$ satisfies $%
\left\vert f(z)\right\vert =M(r,f)$ and $\left\vert z\right\vert =r\notin E$,%
\begin{equation}
f^{(i)}(z)=\left( \frac{\nu (r,f)}{z}\right) ^{i}(1+o(1))f(z)\text{, }%
i=0,...,k\text{.}  \label{9}
\end{equation}%
Now by substituting (\ref{9}) into (\ref{1}), we get that%
\begin{equation*}
\nu (r,f)^{k}+zA_{k-1}(z)\nu (r,f)^{k-1}(1+o(1))~\ \ \ \ \ \ \ \ \ \ \ \ \ \
\ \ 
\end{equation*}%
\begin{equation*}
+\cdots +z^{k-1}A_{1}(z)\nu (r,f)(1+o(1))+z^{k}A_{0}(z)(1+o(1))=0\text{.}
\end{equation*}%
The definition of $(\alpha ,\beta ,\gamma )$-order, Definition \ref{d1} and
Proposition \ref{p1} yields that for any given $\varepsilon >0$ there exists 
$r_{0}>1$, such that for all $r\geq r_{0}$,%
\begin{equation*}
M(r,A_{j})<\exp ^{[2]}\left\{ \alpha ^{-1}((\digamma +\varepsilon )\beta
\left( \log \gamma \left( r\right) \right) )\right\} \text{, }j=0,...,k-1%
\text{.}
\end{equation*}%
By Lemma \ref{l5} we get that,%
\begin{eqnarray*}
\nu (r,f) &\leq &1+\underset{0\leq j\leq k-1}{\max }\left\vert
z^{k-j}A_{j}(z)(1+o(1))\right\vert  \\
&\leq &1+\underset{0\leq j\leq k-1}{\max }\left( 2r^{k-j}\exp ^{[2]}\left\{
\alpha ^{-1}((\digamma +\varepsilon )\beta \left( \log \gamma \left(
r\right) \right) )\right\} \right)  \\
&\leq &1+2r^{k}\exp ^{[2]}\left\{ \alpha ^{-1}((\digamma +\varepsilon )\beta
\left( \log \gamma \left( r\right) \right) )\right\}  \\
&\leq &\exp ^{[2]}\left\{ \alpha ^{-1}((\digamma +2\varepsilon )\beta \left(
\log \gamma \left( r\right) \right) )\right\} \text{, }r\notin E\text{.}
\end{eqnarray*}%
It follows from \cite[p. 36-37]{gl} that%
\begin{eqnarray*}
T(r,f) &\leq &\log M(r,f)\leq \log \mu (r,f)+\log (\nu (2r,f)+2) \\
&\leq &\nu (r,f)\log r+\log (2\nu (2r,f)) \\
&\leq &\exp ^{[2]}\left\{ \alpha ^{-1}((\digamma +2\varepsilon )\beta \left(
\log \gamma \left( r\right) \right) )\right\} \log r+\log (2\exp
^{[2]}\left\{ \alpha ^{-1}((\digamma +2\varepsilon )\beta (\log \gamma
(2r)))\right\} ) \\
&\leq &\exp ^{[2]}\left\{ \alpha ^{-1}((\digamma +3\varepsilon )\beta \left(
\log \gamma \left( r\right) \right) )\right\} +\log 2+\exp \left\{ \alpha
^{-1}((\digamma +2\varepsilon )\beta (\log \gamma (2r)))\right\}  \\
&\leq &\exp ^{[2]}\left\{ \alpha ^{-1}((\digamma +4\varepsilon )\beta \left(
\log \gamma \left( r\right) \right) )\right\} \text{.}
\end{eqnarray*}%
This implies that $\Upsilon \leq \digamma $. Thus Theorem \ref{t1} follows.%
\newline

\textbf{Proof of Theorem \ref{t2}}. By the assumption there exist two
numbers $\lambda _{1}$ and $\lambda $ such that $\sigma _{(\alpha ,\beta
,\gamma )}[A_{m}]\geq \lambda $ and $\sigma _{(\alpha ,\beta ,\gamma
)}[A_{l}]\leq \lambda _{1}<\lambda $ for $l=m+1,...,k-1$.

\qquad Let $f_{1},...,f_{m+1}$ be linearly independent solutions of (\ref{1}%
) such that $\sigma _{(\alpha (\log ),\beta ,\gamma )}[f_{i}]<\lambda $, $%
i=1,...,m+1$. If $m=k-1$, then all $f_{1},...,f_{k}$ are of $\sigma
_{(\alpha (\log ),\beta ,\gamma )}[f_{i}]<\lambda $, this contradict with
the Theorem \ref{t1}. Hence, $m<k-1$. Applying the order reduction procedure
as in the proof of Theorem \ref{t1}, and using the notation $\nu _{0}$
instead of $f$ with $A_{0,0},...,A_{0,k-1}$ instead of $A_{0},...,A_{k-1}$,
and on the basis of general reduction step, we obtain an equation of type%
\begin{equation}
\nu _{j}^{(k-j)}+A_{j,k-j-1}(z)\nu _{j}^{(k-j-1)}+\cdots +A_{j,0}(z)\nu
_{j}=0\text{, }j=1,...,k-1\text{,}  \label{10}
\end{equation}%
where%
\begin{equation}
A_{j,l}=A_{j-1,l+1}+\underset{n=1}{\overset{k-l-j}{\sum }}\binom{l+1+n}{n}%
A_{j-1,l+1+n}\frac{\nu _{j-1,1}^{(n)}}{\nu _{j-1,1}}\text{,}  \label{11}
\end{equation}%
and the functions%
\begin{equation}
\nu _{j,l}(z)=\frac{d}{dz}\left( \frac{\nu _{j-1,l+1}(z)}{\nu _{j-1,1}(z)}%
\right) \text{, }l=1,...,k-j\text{, }\nu _{0}(z)=f(z)\text{, }\nu _{j}(z)=%
\frac{d}{dz}\left( \frac{\nu _{j-1}(z)}{\nu _{0,j-1}(z)}\right) \text{,} 
\notag
\end{equation}%
determine at each reduction step a solution base of (\ref{10}) in terms of
the preceding solution base. We may express (\ref{1}) and the $m-$th
reduction steps by the following table. The rows correspond to (\ref{10})
for $\nu _{0}(z),...,\nu _{m}(z)$, i.e., the first row corresponds to (\ref%
{1}), and columns from $k$ to $0$ give the coefficients of these equations,
while the last column lists those solutions with $\sigma _{(\alpha (\log
),\beta ,\gamma )}[f]<\lambda $.

\begin{center}
$\fbox{$%
\begin{array}[t]{ccccccccc}
~~\hspace{0.25in}~~ & ~~\ ~\ k & k-1~ & \cdot ~\ \  & \mathbf{m~\ \ \ \ \ \ }
& m-1~\ \  & \cdot & ~\ \ 0~\ \  & \sigma _{(\alpha (\log ),\beta ,\gamma
)}[f]<\lambda%
\end{array}%
~$}$

\frame{$%
\begin{array}[t]{ccccccccc}
v_{0} & 1 & A_{0,k-1} & \cdot & \mathbf{A}_{0,m} & A_{0,m-1} & \cdot & 
A_{0,0} & v_{0,1},...,v_{0,m+1} \\ 
v_{1} &  & 1 & \cdot & A_{1,m} & \mathbf{A}_{1,m-1} & \cdot & A_{1,0} & 
v_{1,1},...,v_{1,m} \\ 
\cdot &  &  &  & \cdot & \cdot & \cdot & \cdot & \cdot \\ 
\cdot &  &  &  & \cdot & \cdot & \cdot & \cdot & \cdot \\ 
\cdot &  &  &  & \cdot & \cdot & \cdot & \cdot & \cdot \\ 
v_{m-1} &  &  &  & A_{m-1,m} & A_{m-1,m-1} & \cdot & A_{m-1,0} & 
v_{m-1,1},v_{m-1,2} \\ 
v_{m} &  &  &  & A_{m,m} & A_{m,m-1} & \cdot & \mathbf{A}_{m,0} & v_{m,1}%
\end{array}%
$}
\end{center}

By Lemma \ref{l1} and (\ref{11}), we see that in the second row,
corresponding to the first reduction step,%
\begin{equation*}
m(r,A_{1,l})=O\left( \exp \left\{ \alpha ^{-1}((\lambda _{1}+\varepsilon
)\beta \left( \log \gamma \left( r\right) \right) )\right\} \right) ,\text{ }%
l=m,...,k-2,\text{ }r\notin F\text{,}
\end{equation*}%
while $\lambda _{1}+\varepsilon <\lambda $ and%
\begin{equation*}
m(r,A_{1,m-1})\neq O\left( \exp \left\{ \alpha ^{-1}((\lambda
_{1}+\varepsilon )\beta \left( \log \gamma \left( r\right) \right) )\right\}
\right) \text{, }r\notin F\text{.}
\end{equation*}%
Similarly in each reduction step of (\ref{11}) implies that%
\begin{equation}
m(r,A_{j,l})=O\left( \exp \left\{ \alpha ^{-1}((\lambda _{1}+\varepsilon
)\beta \left( \log \gamma \left( r\right) \right) )\right\} \right) \text{, }%
r\notin F\text{,}  \label{12}
\end{equation}%
when $l=m+1-j,...,k-(j+1)$, i.e., for all coefficients to the left from the
bold face coefficient $A_{j,m-j}$, while for $j=1,...,m$,%
\begin{equation*}
m(r,A_{j,m-1})\neq O\left( \exp \left\{ \alpha ^{-1}((\lambda
_{1}+\varepsilon )\beta \left( \log \gamma \left( r\right) \right) )\right\}
\right) \text{, }r\notin F\text{.}
\end{equation*}%
In particular, 
\begin{equation*}
m(r,A_{m,0})\neq O\left( \exp \left\{ \alpha ^{-1}((\lambda +\varepsilon
)\beta \left( \log \gamma \left( r\right) \right) )\right\} \right) \text{, }%
r\notin F\text{.}
\end{equation*}%
Applying Lemma \ref{l6} to the coefficient $A_{m,0}$ with the constant $%
\lambda $, and obtain that%
\begin{equation}
T(r,A_{m,0})>\exp \left\{ \alpha ^{-1}(\left( \lambda +\varepsilon \right)
\beta \left( \log \gamma \left( r\right) \right) )\right\} \text{, }%
r\rightarrow +\infty \text{, }r\in I\text{.}  \label{13}
\end{equation}%
On the other hand, after the $m-$th reduction step, by (\ref{11}), (\ref{12}%
) and Lemma \ref{l1}, we have%
\begin{equation}
A_{m,0}=-\frac{\nu _{m,1}^{(k-m)}}{\nu _{m,1}}-A_{m,k-m-1}\frac{\nu
_{m,1}^{(k-m-1)}}{\nu _{m,1}}-\cdots -A_{m,1}\frac{\nu _{m,1}^{\prime }}{\nu
_{m,1}}\text{.}  \notag
\end{equation}%
That implies that%
\begin{equation*}
m(r,A_{m,0})=O\left( \exp \left\{ \alpha ^{-1}((\lambda _{1}+\varepsilon
)\beta \left( \log \gamma \left( r\right) \right) )\right\} \right) \text{, }%
r\notin F\text{.}
\end{equation*}%
Since $\sigma _{(\alpha ,\beta ,\gamma )}[\nu _{m,1}]<\lambda _{1}$, in view
of Proposition \ref{p2}, Proposition \ref{p4} and Lemma \ref{l3}, we obtain
that%
\begin{equation*}
N(r,A_{m,0})=O\left( \exp \left\{ \alpha ^{-1}((\lambda _{1}+\varepsilon
)\beta \left( \log \gamma \left( r\right) \right) )\right\} \right) \text{, }%
r\notin F\text{.}
\end{equation*}%
Therefore, 
\begin{equation*}
T(r,A_{m,0})=O\left( \exp \left\{ \alpha ^{-1}((\lambda _{1}+\varepsilon
)\beta \left( \log \gamma \left( r\right) \right) )\right\} \right) \text{, }%
r\notin F\text{.}
\end{equation*}%
By using Lemma \ref{l2} for $d=2$, $\gamma (2r)\leq 2\gamma (r)$ and $\beta
(x+O(1))=(1+o(1))\beta (x)$ as $x\rightarrow +\infty $, for all sufficiently
large $r$, we get that 
\begin{equation*}
T(r,A_{m,0})=O\left( \exp \left\{ \alpha ^{-1}((\lambda _{1}+\varepsilon
)\beta (\log \gamma (2r)))\right\} \right)
\end{equation*}%
\begin{eqnarray}
T(r,A_{m,0}) &\leq &O\left( \exp \left\{ \alpha ^{-1}((\lambda
_{1}+\varepsilon )\left( 1+o\left( 1\right) \right) \beta (\log \gamma
(r)))\right\} \right)  \notag \\
&\leq &\exp \left\{ \alpha ^{-1}((\lambda _{1}+2\varepsilon )\beta \left(
\log \gamma \left( r\right) \right) )\right\} .  \label{14}
\end{eqnarray}%
Since $\varepsilon >0$ satisfies $\lambda _{1}+\varepsilon <\lambda ,$ by (%
\ref{13}) and (\ref{14}), we obtain a contradiction and this proves the
assertion. Hence, there exist at most $m$ linearly independent solutions of (%
\ref{1}) with $\sigma _{(\alpha (\log ),\beta ,\gamma )}[f]<\lambda $.%
\newline

\textbf{Proof of Theorem \ref{t3}}. Let $f$ be a nontrivial solution of (\ref%
{1}). We denote $\sigma _{(\alpha (\log ),\beta ,\gamma )}[f]=\sigma _{1}$
and $\sigma _{(\alpha ,\beta ,\gamma )}[A_{0}]=\sigma _{0}$. The inequality $%
\sigma _{0}\leq \sigma _{1}$ follows from Theorem \ref{t2} when $m=0$ and $%
\lambda =\sigma _{0}$.

\qquad To prove the conserve inequality, by Lemma \ref{l8} for $p=1$ and the
definition of $(\alpha ,\beta ,\gamma )$-order for any given $\varepsilon >0$
and sufficiently large $r$%
\begin{eqnarray*}
m(r,f) &\leq &C\left( \underset{j=0}{\overset{k-1}{\sum \text{ }}}\underset{0%
}{\overset{2\pi }{\int }}\text{ }\underset{0}{\overset{r}{\int }}\left\vert
A_{j}(se^{i\theta })\right\vert ^{\frac{1}{k-j}}dsd\theta +1\right) \\
&\leq &C\text{ }\left( k\underset{0\leq j\leq k-1}{\max }\text{ }\underset{0}%
{\overset{2\pi }{\int }}\text{ }\underset{0}{\overset{r}{\int }}\left\vert
A_{j}(se^{i\theta })\right\vert ^{\frac{1}{k-j}}dsd\theta +1\right) \\
&\leq &C_{1}\text{ }\underset{0\leq j\leq k-1}{\max }\text{ }\underset{0}{%
\overset{r}{\int }}\text{ }\left( \exp ^{[2]}\left\{ (\alpha ^{-1}((\sigma
_{0}+\varepsilon )\beta (\log \gamma (s)))\right\} \right) ^{\frac{1}{k-j}}ds
\\
&\leq &C_{1}\text{ }\underset{0}{\overset{r}{\int }}\exp ^{[2]}\left\{
\alpha ^{-1}((\sigma _{0}+\varepsilon )\beta (\log \gamma (s)))\right\} ds \\
&\leq &C_{1}\text{ }r\text{ }\exp ^{[2]}\left\{ \alpha ^{-1}((\sigma
_{0}+\varepsilon )\beta \left( \log \gamma \left( r\right) \right) )\right\}
\\
&\leq &\exp ^{[2]}\left\{ \alpha ^{-1}((\sigma _{0}+2\varepsilon )\beta
\left( \log \gamma \left( r\right) \right) )\right\} \text{,}
\end{eqnarray*}%
where $C_{1}>0$ is some constant. Therefore%
\begin{equation*}
\frac{\alpha (\log ^{[2]}T(r,f))}{\beta \left( \log \gamma \left( r\right)
\right) }\leq \sigma _{0}+2\varepsilon \text{.}
\end{equation*}%
It is implies that $\sigma _{1}\leq \sigma _{0}$, and then the proof is
complete.\newline

\textbf{Proof of Theorem \ref{t4}}. Suppose that $f$ is a nontrivial
solution of (\ref{1}). From (\ref{1}), we can write that%
\begin{equation}
\left\vert A_{0}(z)\right\vert \leq \left\vert \frac{f^{(k)}(z)}{f(z)}%
\right\vert +\left\vert A_{k-1}(z)\right\vert \left\vert \frac{f^{(k-1)}(z)}{%
f(z)}\right\vert +\cdots +\left\vert A_{1}(z)\right\vert \left\vert \frac{%
f^{\prime }(z)}{f(z)}\right\vert \text{.}  \label{15}
\end{equation}%
If%
\begin{equation*}
\max \{\sigma _{(\alpha ,\beta ,\gamma )}[A_{j}],j=1,...,k-1\}<\sigma
_{(\alpha ,\beta ,\gamma )}[A_{0}]=\sigma _{0}<+\infty \text{,}
\end{equation*}%
then by Theorem \ref{t3}, we obtain that 
\begin{equation*}
\sigma _{(\alpha (\log ),\beta ,\gamma )}[f]=\sigma _{(\alpha ,\beta ,\gamma
)}[A_{0}].
\end{equation*}%
Suppose that%
\begin{equation*}
\max \{\sigma _{(\alpha ,\beta ,\gamma )}[A_{j}],j=1,...,k-1\}=\sigma
_{(\alpha ,\beta ,\gamma )}[A_{0}]=\sigma _{0}<+\infty
\end{equation*}%
and%
\begin{equation*}
\max \{\tau _{(\alpha ,\beta ,\gamma ),M}[A_{j}]:\sigma _{(\alpha ,\beta
,\gamma )}[A_{j}]=\sigma _{(\alpha ,\beta ,\gamma )}[A_{0}]>0\}<\tau
_{(\alpha ,\beta ,\gamma ),M}[A_{0}]=\tau _{M}<+\infty \text{.}
\end{equation*}%
First we prove that 
\begin{equation*}
\sigma _{1}=\sigma _{(\alpha (\log ),\beta ,\gamma )}[f]\geq \sigma
_{(\alpha ,\beta ,\gamma )}[A_{0}]=\sigma _{0}.
\end{equation*}%
By assumption there exists a set $K\subseteq \{1,2,...,k-1\}$ such that%
\begin{equation*}
\sigma _{(\alpha ,\beta ,\gamma )}[A_{j}]=\sigma _{(\alpha ,\beta ,\gamma
)}[A_{0}]=\sigma _{0},\text{ }j\in K
\end{equation*}%
and%
\begin{equation*}
\sigma _{(\alpha ,\beta ,\gamma )}[A_{j}]<\sigma _{(\alpha ,\beta ,\gamma
)}[A_{0}],\text{ }j\in \{1,...,k-1\}\setminus K\text{.}
\end{equation*}%
Thus, we choose $\lambda _{1}$ and $\lambda _{2}$ satisfying%
\begin{equation*}
\max \{\tau _{(\alpha ,\beta ,\gamma ),M}[A_{j}]:j\in K\}<\lambda
_{1}<\lambda _{2}<\tau _{(\alpha ,\beta ,\gamma ),M}[A_{0}]=\tau _{M}\text{.}
\end{equation*}%
For sufficiently large $r$,%
\begin{equation}
\left\vert A_{j}(z)\right\vert \leq \exp ^{[2]}\left\{ \alpha ^{-1}(\log
(\lambda _{1}(\exp \left( \beta \left( \log \gamma \left( r\right) \right)
\right) )^{\sigma _{0}}))\right\} ,\text{ }j\in K  \label{16}
\end{equation}%
and%
\begin{equation*}
\left\vert A_{j}(z)\right\vert \leq \exp ^{[2]}\left\{ \alpha ^{-1}(\log
(\lambda _{1}(\exp \left( \beta \left( \log \gamma \left( r\right) \right)
\right) )^{\xi }))\right\}
\end{equation*}%
\begin{equation}
\leq \exp ^{[2]}\left\{ \alpha ^{-1}(\log (\lambda _{1}(\exp \left( \beta
\left( \log \gamma \left( r\right) \right) \right) )^{\sigma _{0}}))\right\}
,\text{ }j\in \{1,...,k-1\}\setminus K\text{,}  \label{17}
\end{equation}%
where $0<\xi <\sigma _{0}$. By Lemma \ref{l7}, there exists a set $I\subset
\left( 1,+\infty \right) $ with infinite logarithmic measure, such that for
all $r\in I$,%
\begin{equation}
\left\vert A_{0}(z)\right\vert >\exp ^{[2]}\left\{ \alpha ^{-1}(\log
(\lambda _{2}(\exp \left( \beta \left( \log \gamma \left( r\right) \right)
\right) )^{\sigma _{0}}))\right\} \text{.}  \label{18}
\end{equation}%
By Lemma \ref{l9}, there exist a constant $B>0$ and a set $E\subset \left(
1,+\infty \right) $ having finite logarithmic measure, such that for all $z$
satisfying $\left\vert z\right\vert =r\notin E\cup \left[ 0,1\right] $,%
\begin{equation*}
\left\vert \frac{f^{(j)}(z)}{f(z)}\right\vert \leq B[T(2r,f)]^{k+1}\text{, }%
j=1,2,...,k\text{.}
\end{equation*}%
By Definition \ref{d1} and Proposition \ref{p1}, for any given 
\begin{equation*}
\varepsilon \in \left( 0,\max \left\{ \frac{\lambda _{2}-\lambda _{1}}{2}%
,\sigma _{0}-\sigma _{1}\right\} \right)
\end{equation*}%
and sufficiently large $\left\vert z\right\vert =r\notin E\cup \left[ 0,1%
\right] $, we get that%
\begin{eqnarray}
\left\vert \frac{f^{(j)}(z)}{f(z)}\right\vert &\leq &B[T(2r,f)]^{k+1}  \notag
\\
&\leq &B\left( \exp ^{[2]}\left\{ \alpha ^{-1}((\sigma _{1}+\varepsilon
)\beta (\log \gamma (2r)))\right\} \right) ^{k+1}\text{, }j=1,...,k\text{.}
\label{19}
\end{eqnarray}%
Hence substituting (\ref{16}), (\ref{17}), (\ref{18}) and (\ref{19}) into (%
\ref{15}), for sufficiently large $\left\vert z\right\vert =r\in I\setminus
(E\cup \lbrack 0,1])$,%
\begin{eqnarray}
&&\exp ^{[2]}\left\{ \alpha ^{-1}(\log (\lambda _{2}(\exp \left( \beta
\left( \log \gamma \left( r\right) \right) \right) )^{\sigma _{0}}))\right\}
\notag \\
&\leq &kB\exp ^{[2]}\left\{ \alpha ^{-1}(\log (\lambda _{1}(\exp \left(
\beta \left( \log \gamma \left( r\right) \right) \right) )^{\sigma
_{0}}))\right\} \cdot \left( \exp ^{[2]}\left\{ \alpha ^{-1}((\sigma
_{1}+\varepsilon )\beta (\log \gamma (2r)))\right\} \right) ^{k+1}  \notag \\
&\leq &\exp ^{[2]}\left\{ \alpha ^{-1}(\log ((\lambda _{1}+2\varepsilon
)(\exp \left( \beta \left( \log \gamma \left( r\right) \right) \right)
)^{\sigma _{0}}))\right\} \text{.}  \label{20}
\end{eqnarray}%
Obviously, $I\setminus (E\cup \lbrack 0,1])$ is of infinite logarithmic
measure. By (\ref{20}), there exists a sequence of points $\{\left\vert
z_{n}\right\vert \}=\{r_{n}\}\subset I\setminus (E\cup \lbrack 0,1])$
tending to $+\infty $ such that%
\begin{eqnarray*}
&&\exp ^{[2]}\left\{ \alpha ^{-1}(\log (\lambda _{2}(\exp (\beta (\log
\gamma (r_{n}))))^{\sigma _{0}}))\right\} \\
&\leq &\exp ^{[2]}\left\{ \alpha ^{-1}(\log ((\lambda _{1}+2\varepsilon
)(\exp (\beta (\log \gamma (r_{n}))))^{\sigma _{0}}))\right\} \text{.}
\end{eqnarray*}%
From above we get that $\lambda _{1}\geq \lambda _{2}$. This contradiction
implies 
\begin{equation*}
\sigma _{(\alpha (\log ),\beta ,\gamma )}[f]\geq \sigma _{(\alpha ,\beta
,\gamma )}[A_{0}]\text{.}
\end{equation*}%
On the other hand, by Lemma \ref{l10}, we get that%
\begin{equation*}
\sigma _{(\alpha (\log ),\beta ,\gamma )}[f]\leq \max \{\sigma _{(\alpha
,\beta ,\gamma )}[A_{j}]:j=0,1,...,k-1\}=\sigma _{(\alpha ,\beta ,\gamma
)}[A_{0}]\text{.}
\end{equation*}%
Hence every nontrivial solution $f$ of (\ref{1}) satisfies $\sigma _{(\alpha
(\log ),\beta ,\gamma )}[f]=\sigma _{(\alpha ,\beta ,\gamma )}[A_{0}]$. 
%%%%%%%%%%%%%%%%%%%%%%%%%%%%%%%%%%%%%%%%%%%%%%%%%%%%%%%%%%%%%%%%%%%%%%%%%%%%%%%%%%%%
\newline

\end{document}